\documentclass{amsart}
\usepackage{graphics}
\usepackage{graphicx}
\usepackage{amsfonts,amssymb,amsopn,amscd,amsmath}
\usepackage[shortlabels]{enumitem}
\usepackage[usenames,dvipsnames]{xcolor}
\usepackage{tikz}
\usepackage{algpseudocode}
\usetikzlibrary{matrix,arrows,calc,cd,decorations.pathmorphing}
\usepackage{mathrsfs}
\usepackage{multirow}

\usepackage[utf8]{inputenc}
\allowdisplaybreaks

\usepackage[usenames,dvipsnames]{xcolor}
\definecolor{darkblue}{rgb}{0.0, 0.0, 0.55}
\definecolor{bordeaux}{rgb}{0.34, 0.01, 0.1}

\usepackage[pagebackref,colorlinks,linkcolor=bordeaux,citecolor=darkblue,urlcolor=black,hypertexnames=true]{hyperref}

 \newtheorem{theorem}{Theorem}[section]
 \newtheorem{corollary}[theorem]{Corollary}
 \newtheorem{lemma}[theorem]{Lemma}{\rm}
 \newtheorem{proposition}[theorem]{Proposition}

 \newtheorem{definition}[theorem]{Definition}{\rm}
 \newtheorem{assumption}[theorem]{Assumption}
 \newtheorem{remark}[theorem]{Remark}
 \newtheorem{example}[theorem]{Example}

\newtheorem{algorithm}[theorem]{Algorithm}

\numberwithin{equation}{section}
\newcommand{\op}[1]{\operatorname{#1}}

\DeclareMathOperator{\II}{II}
\DeclareMathOperator{\cdeg}{cdeg}
\DeclareMathOperator{\Sym}{Sym}

\DeclareMathOperator{\Trace}{tr}
\DeclareMathOperator{\supp}{supp}
\DeclareMathOperator{\rank}{rank}
\DeclareMathOperator{\sparse}{sparse}
\DeclareMathOperator{\sohs}{sohs}
\newif\ifcomment
\commentfalse
\commenttrue
\begin{document}
\def\cA{\mathcal A}
\def\cH{\mathcal H}
\def\cM{\mathcal M}
\def\red{\color{red}}
\def\bl{\color{blue}}
\def\ora{\color{orange}}
\def\green{\color{green}}
\def\br{\color{brown}}
\newcommand{\realtofloat}{\mathtt{Real2Float}}
\newcommand{\sparsepop}{\mathtt{SparsePOP}}
\newcommand{\ncsostools}{\mathtt{NCSOStools}}
\newcommand{\ncpoltosdpa}{\mathtt{Ncpol2sdpa}}
\newcommand{\gloptipoly}{\mathtt{Gloptipoly}}
\def\la{\langle}
\def\ra{\rangle}
\def\e{{\rm e}}
\def\x{\mathbf{x}}
\def\by{\mathbf{y}}
\def\bz{\mathbf{z}}
\def\F{\mathcal{F}}
\def\R{\mathbb{R}}
\def\Mbb{\mathbb{M}}
\def\Sbb{\mathbb{S}}
\def\T{\mathbf{T}}
\def\N{\mathbb{N}}
\def\K{\mathbf{K}}
\def\bK{\overline{\mathbf{K}}}
\def\Q{\mathbf{Q}}
\def\M{\mathbf{M}}
\def\O{\mathbf{O}}
\def\C{\mathbf{C}}
\def\P{\mathbf{P}}
\def\Z{\mathbb{Z}}
\def\H{\mathcal{H}}
\def\A{\mathbf{A}}
\def\W{\mathbf{W}}
\def\bfone{\mathbf{1}}
\def\V{\mathbf{V}}
\def\AA{\overline{\mathbf{A}}}
\def\c{\mathbf{C}}
\def\L{\mathbf{L}}
\def\bS{\mathbf{S}}
\def\H{\mathcal{H}}
\def\I{\mathbf{I}}
\def\Y{\mathbf{Y}}
\def\X{\mathbf{X}}
\def\G{\mathbf{G}}
\def\Bbb{\mathbb{B}}
\def\Dbb{\mathbb{D}}
\def\f{\mathbf{f}}
\def\z{\mathbf{z}}
\def\y{\mathbf{y}}
\def\d{\hat{d}}
\def\bx{\mathbf{x}}
\def\y{\mathbf{y}}
\def\h{\mathbf{h}}
\def\v{\mathbf{v}}
\def\u{\mathbf{u}}
\def\g{\mathbf{g}}
\def\w{\mathbf{w}}
\def\b{\mathcal{B}}
\def\a{\mathbf{a}}
\def\q{\mathbf{q}}
\def\u{\mathbf{u}}
\def\s{\mathcal{S}}
\def\cc{\mathcal{C}}
\def\co{{\rm co}\,}
\def\cp{{\rm CP}}
\def\tg{\tilde{f}}
\def\tx{\tilde{\x}}
\def\supmu{{\rm supp}\,\mu}
\def\supnu{{\rm supp}\,\nu}
\def\m{\mathcal{M}}
\def\bR{\mathbf{R}}
\def\om{\mathbf{\Omega}}
\def\c{\mathbf{c}}
\def\s{\mathcal{S}}
\def\k{\mathcal{K}}
\def\la{\langle}
\def\ra{\rangle}
\def\opt{\text{opt}}
\def\cyc{\overset{\text{cyc}}{\sim}}
\def\smileL{\overset{\smallsmile}{L}}
\def\blambda{{\boldsymbol{\lambda}}}
\def\bsigma{{\boldsymbol{\sigma}}}
\def\RX{\R \langle \underline{X} \rangle}
\def\uX{\underline X}
\newcommand{\RXI}[1]{\R \langle \underline{X}(I_{#1}) \rangle }
\def\SigmaX{\Sigma \langle \underline{X} \rangle}
\newcommand{\SigmaXI}[1]{\Sigma \langle \underline{X}(I_{#1}) \rangle }
\def\SymRX{\Sym \RX}
\newcommand{\SymRXI}[1]{\Sym \R \langle \underline{X}(I_{#1}) \rangle }
\def\ov{\overline{o}}
\def\und{\underline{o}}
\newcommand{\victor}[1]{\todo[inline,color=purple!30]{VM: #1}}
\newcommand{\victorshort}[1]{\todo[inline,color=purple!30]{VM: #1}}
\newcommand{\igor}[1]{\Ig{#1}}
\newcommand{\igorshort}[1]{\todo[color=brown!30]{IK: #1}}
\newcommand{\janez}[1]{{\color{red} Janez: #1}}
\newcommand{\janezshort}[1]{\todo[color=red!30]{TdW: #1}}

\colorlet{commentcolour}{green!50!black}
\newcommand{\comment}[3]{%
\ifcomment%
	{\color{#1}\bfseries\sffamily(#3)%
	}%
	\marginpar{\textcolor{#1}{\hspace{3em}\bfseries\sffamily #2}}%
	\else%
	\fi%
}
\newcommand{\Ig}[1]{
	\comment{magenta}{I}{#1}
}
\newcommand{\Ja}[1]{
	\comment{green}{J}{#1}
}
\newcommand{\Vi}[1]{
	\comment{blue}{V}{#1}
}
\newcommand{\idea}[1]{\textcolor{red}{#1(?)}}

\newcommand{\Expl}[1]{
	{\tag*{\text{\small{\color{commentcolour}#1}}}%
	}
}

\renewcommand{\algorithmicrequire}{\textbf{Input:}}
\renewcommand{\algorithmicensure}{\textbf{Output:}}
\newcommand{\sparsegns}{\texttt{SparseGNS}}
\newcommand{\sparseeiggns}{\texttt{SparseEigGNS}}
\newcommand{\eigmin}{\texttt{NCeigMin}}
\newcommand{\eigminsparse}{\texttt{NCeigMinSparse}}
\def\nsdp{n_{\text{sdp}}}
\def\msdp{m_{\text{sdp}}}
\newcommand{\newncsparse}[1]{{{\color{blue}#1}}}
\newcommand{\newnc}[1]{{{\color{black}#1}}}
\title[Sparse Noncommutative Polynomial Optimization]{Sparse  Noncommutative Polynomial Optimization}

\author{Igor Klep \and Victor Magron \and Janez Povh}

\address{Igor Klep: Faculty of Mathematics and Physics,  University of Ljubljana, Slovenia}
\email{igor.klep@fmf.uni-lj.si}
\thanks{IK was supported by the 
Slovenian Research Agency grants J1-8132, N1-0057 and P1-0222. Partially supported 
by the 
Marsden Fund Council of the Royal Society of New Zealand.}
\address{Victor Magron: LAAS-CNRS and Institute of Mathematics, Toulouse, France}
\email{vmagron@laas.fr}
\thanks{VM was supported by the FMJH Program PGMO (EPICS project) and  EDF, Thales, Orange et Criteo, as well as from the Tremplin ERC Stg Grant ANR-18-ERC2-0004-01 (T-COPS project). }
\address{Janez Povh: Faculty of Mechanical Engineering, University of Ljubljana, Slovenia}
\email{janez.povh@lecad.fs.uni-lj.si}
\thanks{JP was supported bt the Slovenian Research Agency program P2-0256 and grants J1-8132, J1-8155, N1-0057 and N1-0071.}
\date{}

\begin{abstract}
This article focuses on optimization of polynomials in noncommuting variables, while taking into account sparsity in the input data. 
A converging hierarchy of semidefinite relaxations for eigenvalue and trace optimization is provided. 
This hierarchy is  a noncommutative analogue of results due to Lasserre [SIAM J. Optim. 17(3) (2006), pp. 822--843] and Waki et al. [SIAM J. Optim. 17(1) (2006), pp. 218--242]. 
The Gelfand-Naimark-Segal (GNS) construction is applied to extract optimizers if flatness and irreducibility conditions are satisfied.
Among the main techniques used are amalgamation results from operator algebra.
The theoretical results are utilized to compute lower bounds on minimal eigenvalue of noncommutative polynomials from the literature.
\end{abstract}



\maketitle

\section{Introduction}
\label{sec:intro}
The goal of this article is to handle a specific class of sparse  polynomial optimization problems with noncommuting variables (e.g.,~polynomials in matrices). 
Applications of interest include control theory and linear systems in engineering~\cite{skelton1997unified}, quantum theory and quantum information science~\cite{navascues2008convergent}. 
For example, in the latter context, noncommutative polynomial optimization   provides upper bounds on the maximum violation level of Bell inequalities~\cite{pal2009quantum}.
Further motivations relate to the generalized Lax conjecture~\cite{lax1957differential}, with computational proof attempts relying on noncommutative sums of squares (in Clifford algebras)~\cite{netzer2014hyperbolic}.
The problem of verifying noncommutative polynomial inequalities has also occurred in a conjecture formulated by Bessis, Moussa and Villani (BMV) in 1975~\cite{bessis1975monotonic}, and restated by Lieb and Seiringer~\cite{lieb2004equivalent}. 
The initial conjecture boils down to verifying that the (univariate) polynomial $t \mapsto \Trace(A+tB)^m$ has only nonnegative coefficients, for all positive semidefinite matrices $A$ and $B$, and all $m \in \N$.
Even though the BMV conjecture has been established by Stahl~\cite{stahl2013proof} for all $m$, one can rely on computational proofs for a fixed value of $m$.
Schweighofer and the first author derived a computational proof~\cite{klep2008sums} of the conjecture for $m \leq 13$. 
Recently, noncommutative polynomial optimization has been used in~\cite{Gribling18} to study optimization problems related to bipartite quantum correlations, and in~\cite{Gribling19} to derive hierarchies of lower bounds for several matrix factorization ranks, including nonnegative rank, positive semidefinite rank as well as their symmetric analogues.

In the commutative setting, \emph{polynomial optimization} focuses on minimizing or maximizing a polynomial over a \emph{semialgebraic} set, that is, a set defined by a finite conjunction/disjunction of polynomial inequalities. 
In general, computing the {exact} solution of a polynomial optimization problem is an NP-hard problem~\cite{Laurent:Survey}. 
In practice, one can at least compute an approximation of the solution by considering a relaxation of the problem.  
In the seminal 2001 paper~\cite{Las01sos}, Lasserre  introduced a nowadays famous hierarchy of relaxations, called the  {\emph{moment-sums of squares hierarchy}} allowing us to obtain a converging sequence of lower bounds for the minimum of a polynomial over a compact semialgebraic set. 
Each lower bound is computed by solving a semidefinite program (SDP). 
In SDP, one optimizes a linear function under a linear matrix inequality constraint. 
SDP itself appears in a wide range of applications (combinatorial optimization~\cite{laurent2005semidefinite}, control theory~\cite{BEFB94}, matrix completion~\cite{Lau09}, etc.) and can be solved efficiently (up to a few thousand optimization variables) by freely available software, e.g.~SeDuMi~\cite{Sturm98usingsedumi}, SDPT3~\cite{SDPT3}, SDPA~\cite{sdpa} or Mosek~\cite{moseksoft}. 
For optimization problems involving $n$-variate polynomials of degree less than $d$, the size of the matrices involved at step $k \geq d$ of Lasserre's hierarchy of SDP relaxations is proportional to $\binom{n+k}{n}$.
Therefore, the size of the SDP problems arising from the hierarchy grows rapidly.\looseness=-1

For unconstrained problems involving a large number of variables $n$, a remedy consists of reducing the size of the SDP matrices by discarding the monomials which never appear in the support of the SOS decompositions. 
This technique, based on a result by Reznick~\cite{Reznick78},  consists of computing the \emph{Newton polytope} of the input polynomial (that is, the convex hull of the support of this polynomial) and selecting only monomials with support lying in half of this polytope.
For constrained optimization, another workaround is based on  exploiting a potential sparsity/symmetry pattern arising in the input polynomials.
In \cite{Las06} (see also~\cite{Waki06} and the related $\sparsepop$ solver~\cite{WakiKKMS08}), the author derives a sparse version of Putinar's representation~\cite{Putinar1993positive} for polynomials positive on compact semialgebraic sets. 
{See also \cite{grimm2007note} for a simpler proof.}
This variant can be used for cases where the objective function can be written as a sum of polynomials, each of them involving a  small number of variables. 
Sparse polynomial optimization techniques enable us to successfully handle various concrete applications. 
The frameworks~\cite{toms17,toms18}, coming with the $\realtofloat$ software library, rely on these techniques to produce a hierarchy of upper bounds converging to the absolute roundoff error of a numerical program involving polynomial operations. 
In energy networks, it is now possible to compute the solution of large-scale power flow problems with up to thousand variables~\cite{Josz16}.
In ~\cite{tacchi2019exploiting}, the authors derive the sparse analogue of~\cite{HLS09vol} to obtain a hierarchy of upper bounds for the volume of large-scale semialgebraic sets. 
\newnc{Recently, sparse polynomial optimization has been used to bound the Lipschitz constants of ReLU networks \cite{chen} and to handle sparse positive definite forms \cite{sparseReznick}.}
In the same spirit, the symmetry pattern of the  polynomial optimization problem can be exploited~\cite{Riener13Symmetries}.
More recent progress focused on the use of alternative hierarchies, including the so-called {\emph{bounded degree SOS hierarchy} (BSOS)}~\cite{lasserre2017bounded}. 
Here, one represents a positive polynomial as the sum of two terms: an SOS polynomial of a priori fixed degree, and a term lying in the set of {\emph{Krivine-Stengle}} representations~\cite{krivineanneaux}, that is, a combination of positive linear cross-products of the polynomials involved in the set of constraints.
The BSOS hierarchy can handle bigger instances than the standard moment-SOS hierarchy. 
In addition, sparsity can be exploited in the same way as for the sparse SOS hierarchy, which allows us to tackle even larger problems~\cite{weisser2018sparse}.

In the noncommutative context, a given noncommutative polynomial in $n$ variables and degree $d$ is positive semidefinite if and only if it decomposes as a \emph{sum of hermitian squares} (SOHS)~\cite{Helton02,McCullSOS}.  
In practice, an SOHS decomposition can be computed by solving an SDP of size $O(n^d)$, which is even larger than the size of the matrices involved in the commutative case.
SOHS decompositions are also used for constrained optimization, either to minimize eigenvalues or traces of noncommutative polynomial objective functions, under noncommutative polynomial (in)equality constraints. 
The optimal value of such constrained problems can be approximated, as closely as desired, while relying on the noncommutative analogue of Lasserre's hierarchy~\cite{pironio2010convergent,cafuta2012constrained,nctrace}.
The $\ncsostools$~\cite{cafuta2011ncsostools,burgdorf16} library can compute such approximations for optimization problems involving polynomials in noncommuting variables.
By comparison with the commutative case, the size $O(n^k)$ of the SDP matrices at a given step $k$ of the noncommutative hierarchy becomes intractable even faster, typically for $k,n \simeq 6$ on a standard laptop.

A remedy for unconstrained problems is to rely on the adequate noncommutative analogue of the standard Newton polytope method, which is called the \emph{Newton chip method} (see, e.g.~\cite[\S2.3]{burgdorf16}) and can be further improved with the \emph{augmented Newton chip method} (see, e.g.,~\cite[\S2.4]{burgdorf16}), by removing certain terms which can never appear in an SOHS decomposition of a given input.
As in the commutative case, the Newton polytope method cannot be applied for constrained problems.
When one cannot go from step $k$ to step $k+1$ in the hierarchy because of the computational burden, one can always consider matrices indexed by all terms of degree $k$ plus a fixed percentage of terms of degree $k+1$. 
This is used for instance to compute tighter upper bounds for maximum violation levels of Bell inequalities~\cite{pal2009quantum}.
Another trick, implemented in the $\ncpoltosdpa$ library~\cite{wittek2015algorithm}, consists of exploiting simple equality constraints, such as ``$X^2 = Y$'', to derive substitution rules for variables involved in the SDP relaxations. 
Similar substitutions are performed in the commutative case by $\gloptipoly$ 3~\cite{gloptipoly3}.

Apart from such heuristic procedures, there is, to the best of our knowledge, no general method to exploit
additional structure, such as
  sparsity, of (un)constrained noncommutative polynomial optimization problems.

\paragraph{\textbf{Contributions}} 
We state and prove in Section~\ref{sec:sparsePutinar} a sparse variant of the noncommutative version of Putinar's Positivstellensatz, under the same sparsity pattern assumptions as the ones used in the commutative case~\cite{Las06,Waki06};
these conditions are known as the \emph{running intersection property} (RIP) in graph theory~\cite{fukuda2001exploiting,nakata2003exploiting}.
Our proof relies on amalgamation results for operator algebras.
Then, we present in Section~\ref{sec:extract} 
the sparse Gelfand-Naimark-Segal (GNS) construction 
yielding  representations for linear functionals 
positive w.r.t.~sparsity. This allows us
to extract minimizers, providing that flatness and  irreducibility conditions hold.
We rely on this sparse representation to design algorithms performing eigenvalue optimization (Section~\ref{sec:eig}) and trace optimization of noncommutative sparse polynomials (Section~\ref{sec:trace}), both in the unconstrained and constrained case.
Along the way we exhibit an example showing that the Helton-McCullough \cite{Helton02,McCullSOS} Sum of Squares theorem (every positive nc polynomial is a sum of hermitian squares) fails in the sparse setting, see Lemma \ref{lemma:nosparseHelton}.
Finally, we provide in Section~\ref{sec:bench} experimental comparisons between the bounds given by the dense relaxations and the ones produced by our algorithms, currently implemented in the $\ncsostools$  software library.

\section{Notation and Definitions}
\label{sec:prelim}
This section gives the basic definitions and introduces notation used in the rest of the article.
\subsection{Noncommutative polynomials}
\label{sec:prelim_nc}
Let us denote by $\Mbb_n(\R)$ (resp.~$\Sbb_n$) the space of all real (resp.~symmetric) matrices of order $n$, and by $\Sbb_n^k$ the set of $k$-tuples $\underline{A} = (A_1,\dots,A_k)$ of symmetric matrices $A_i$ of order $n$.
{The normalized trace of an $n\times n$ matrix $A=(a_{i,j})_{i,j}$ is given by $\Trace A = \frac{1}{n} \sum_{i=1}^n a_{i,i}$.
Given $A \in \Sbb_n$, we write $A \succeq 0$ (resp.~$A \succ 0$) when $A$ is positive semidefinite (resp.~positive definite), i.e., has only nonnegative (resp.~positive) eigenvalues.}
Let $\I_n$ stands for the identity matrix of order $n$.
For a fixed $n \in \N$, we consider a finite alphabet $X_1,\dots,X_n$ of {symmetric letters} and generate all possible words of finite length in these letters. 
The empty word is denoted by 1. 
The resulting set of words is  $\langle \underline{X} \rangle$, with $\underline{X} = (X_1,\dots, X_n)$. 
We denote by $\RX$ the set of real polynomials in noncommutative variables, abbreviated as \textit{nc polynomials}.
A \textit{monomial} is an element of the form $a_w w$, with $a_w \in \R \backslash \{0\}$ and $w \in \langle \underline{X} \rangle$. 
The \textit{degree} of an nc polynomial $f \in \RX$ is the length of the longest word involved in $f$.
For $d \in \N$, $\langle \underline{X} \rangle_d$ is the set of all words of degree at most $d$.
Let us denote by $\W_d$ the vector of all words of $\langle \underline{X} \rangle_d$ w.r.t.~to the lexicographic order. 
Note that the dimension of $\RX_d$ is the length of $\W_d$, which is $\bsigma(n,d) := \sum_{i=0}^d n^i = \frac{n^{d+1}-1}{n-1}$. 
The set $\RX$ is equipped with the involution $\star$ that fixes $\R \cup \{X_1,\dots,X_n\}$ point-wise and reverses words, so that $\RX$ is the $\star$-algebra freely generated by $n$ symmetric letters $X_1,\dots,X_n$. 
The set of all \textit{symmetric elements} is defined as $\SymRX := \{f \in \RX : f = f^\star  \}$.
~\\
\paragraph{\textbf{Sums of hermitian squares}}
An nc polynomial of the form $g^\star g$ is called a {\em hermitian square}. 
A given $f \in \RX$ is a {\em sum of hermitian squares} (SOHS) if there exist nc polynomials $h_1,\dots,h_r \in \RX$, with $r \in \N$, such that $f = h_1^\star h_1 + \dots + h_r^\star h_r$. 
Let $\SigmaX$ stands for the set of SOHS. We denote by $\SigmaX_{d} \subseteq \SigmaX$ the set of SOHS polynomials of degree at most $2 d$.
We now recall how to check whether a given $f \in \SymRX$ is an SOHS. 
The existing procedure, known as the \emph{Gram matrix method}, relies on the following proposition (see, e.g.,~\cite[\S2.2]{Helton02}):
\begin{proposition}
\label{prop:ncGram}
Assume that $f \in \SymRX_{2 d}$. 
Then $f \in \SigmaX$ if and only if there exists $G_f \succeq 0$ satisfying
\begin{align}
\label{eq:ncGram}
f = \W_d^\star \, G_f \, \W_d \,.
\end{align}
Conversely, given such $G_f \succeq 0$ with rank $r$, one can construct $g_1,\dots,g_r \in \RX_d$ such that $f = \sum_{i=1}^r g_i^\star g_i$.
\end{proposition}
Any symmetric matrix $G_f$ (not necessarily positive semidefinite) satisfying~\eqref{eq:ncGram} is called a \emph{Gram matrix} of $f$.

%
~\\
\paragraph{\textbf{Semialgebraic sets and quadratic modules}}
Given a positive integer $m$ and $S = \{s_1,\dots,s_m \} \subseteq \SymRX$, the {\em semialgebraic} set ${D_S}$ associated to $S$ is defined as follows:
\begin{align}
\label{eq:DS}
{D_S} := \bigcup_{k \in \N} \{ \underline{A} = (A_1,\dots,A_n) \in \Sbb_k^n : s_j(\underline{A}) \succeq 0 \,, \quad j=1 \dots, m  \} \,.
\end{align}
When considering only tuples of $N \times N$ symmetric matrices, we use the notation ${D_S^N} := {D_S} \cap \Sbb_N^n$.
The {\em operator semialgebraic set} ${D_S^\infty}$ is the set of all bounded self-adjoint operators $\underline{A}$ {on a Hilbert space $\mathcal{H}$ endowed with a scalar product $\langle \cdot \mid \cdot \rangle$, making $g(\underline{A})$ a positive semidefinite operator for all $g \in S$, i.e., $\langle g(\underline{A}) v \mid v \rangle \geq 0$, for all $v \in \mathcal{H}$.
We say that a noncommutative polynomial $f$ is positive (denoted by $f \succ 0$) on ${D_S^\infty}$  if for all $\underline{A} \in {D_S^\infty}$ the operator $f(\underline{A})$  is positive definite, i.e., $\langle f(\underline{A}) v \mid v \rangle > 0$, for all nonzero $v \in \mathcal{H}$.
} 
%
The {\em quadratic module} ${\cM(S)}$, generated by $S$, is defined by
\begin{align}
\label{eq:MS}
{\cM(S)}:= \left\{ \sum_{i=1}^K  a_i^\star s_i' a_i : K \in \N \,, a_i \in \RX \,, s_i' \in S \cup \{1\}  \right\} \,.
\end{align}
Given $d \in \N$, the truncated {\em quadratic module} ${\cM(S)_d}$ of order $d$, generated by $S$, is 
\begin{align}
\label{eq:MS2d}
{\cM(S)_d} := \left\{ \sum_{i=1}^K  a_i^\star s_i' a_i : K \in \N \,, a_i \in \RX \,, s_i' \in S \cup \{1\} \,, \deg (a_i^\star s_i' a_i) \leq 2 d  \right\} \,.
\end{align}
{Let $\textbf{1}$ stands for the unit polynomial.}
A quadratic module ${\cM}$ is called {\em archimedean} if for each $ a \in \RX$, there exists $N \in {\R^{\geq 0}}$ such that $N \cdot {\textbf{1}} - a^\star a \in {\cM}$. 
{One can show that
this is equivalent to the existence of an $N \in \R^{\geq 0}$ such that $N \cdot \textbf{1} - \sum_{i=1}^n X_i^2 \in \cM$}. 

The noncommutative analog of Putinar's Positivstellensatz \cite{Putinar1993positive} 
describing noncommutative polynomials positive on ${D_S^\infty}$ with archimedean ${\cM(S)}$
is due to Helton and McCullough:
\begin{theorem}[\protect{\cite[Theorem~1.2]{Helton04}}]
\label{th:densePsatz}
Let $S \cup \{f\} \subseteq \SymRX$ and assume that ${\cM(S)}$ is archimedean.
If $f(\underline{A}) \succ 0$ for all $\underline{A} \in {D_S^\infty}$, then $f \in {\cM(S)}$.
\end{theorem}
\subsection{Sparsity patterns}
\label{sec:prelim_sparse}
Let $I_0 := \{1,\dots,n \}$.
For $p\in \N$ consider $I_1,\dots,I_p \subseteq I_0$ satisfying $\bigcup_{k=1}^p I_k = I_0$. 
Let $n_k$ be the size of $I_k$, for each $k=1,\dots,p$.

We denote by $\langle \underline{X}(I_k) \rangle$ (resp.~$\RXI{k}$) the set of words (resp.~nc polynomials) in the $n_k$ variables $\underline{X}(I_k) = \{X_i : i \in I_k \}$. 
The dimension of $\RXI{k}_d$ is $\bsigma(n_k,d) = \frac{n_k^{d+1}-1}{n_k-1}$. 
Note that $\RXI{0} = \RX$.
We also define $\SymRXI{k} := \SymRX \cap \RXI{k}$, let $\SigmaXI{k}$ stands for the set of SOHS in $\RXI{k}$ and we denote by $\SigmaXI{k}_{d}$ the restriction of $\SigmaXI{k}$ to nc polynomials of degree at most $2 d$.
\if{
Let $\mathcal{I}_k$ be the set of all subsets of $I_k$, for all $k=1,\dots,p$. 
For $u \in \langle \underline{X} \rangle$,  $\supp{u}$ denotes the support of $u$, that is, the set of all $i \in \{1,\dots,n\}$ for which the letter $X_i$ appears in the word $u$.
As an example, if one considers $n = 6$ and $u = X_1^2 X_5  X_6$, then $\supp{u} = \{1,5,6 \}$.
\fi
In the sequel, we will rely on two specific assumptions. The first one is as follows:
\begin{assumption}[Boundedness]
\label{hyp:sparsity}
Let ${D_S}$ be as in~\eqref{eq:DS}. 
There is ${N} \in \R^{> 0}$ such that $\sum_{i=1}^n X_i^2 \preceq N \cdot \textbf{1}$, for all $\underline{X} = (X_1,\dots,X_n) \in {D_S^\infty}$. 
\end{assumption}
Then, Assumption~\ref{hyp:sparsity} implies that $\sum_{j \in I_k} X_j^2 \preceq {N} \cdot {\textbf{1}}$, for all $k=1,\dots,p$. 
Thus we define 
\begin{align}
\label{eq:additional}
s_{m+k} := {N \cdot \textbf{1}} - \sum_{j \in I_k} X_j^2 \,, \quad k=1,\dots,p \,,
\end{align}
and set $m' = m + p$ in order to describe the same set ${D_S}$ again as:
\begin{align}
\label{eq:newDS}
{D_S := \bigcup_{k \in \N} \{ \underline{A} \in \Sbb_k^n } : s_j(\underline{A}) \succeq 0, \quad j=1,\dots,m' \} \,,
\end{align}
{as well as the operator semialgebraic set ${D_S^\infty}$.}

The second assumption is as follows:
\begin{assumption}[RIP]
\label{hyp:sparsityRIP}
Let ${D_S}$ be as in~\eqref{eq:newDS} and let $f \in \RX$. 
The index set $J := \{1, \dots, m' \}$ is partitioned into $p$ disjoint sets $J_1,\dots,J_p$ and the two collections $\{I_1,\dots,I_p\}$ and $\{J_1,\dots,J_p \}$ satisfy:
\begin{enumerate}[\rm (i)]
\item For all $j \in J_k$, $g_j \in \SymRXI{k}$.
\item The objective function can be decomposed as $f = f_1 + \dots + f_p$, with $f_k \in \RXI{k}$, for all $k=1, \dots,p$.
\item The running intersection property (RIP) holds, i.e., for all $k=1,\dots,p-1$, one has
\begin{align}
\label{eq:RIP}
I_{k+1} \cap \bigcup_{j \leq k} I_j \subseteq I_\ell \quad \text{for some } \ell \leq k \,.
\end{align}
\end{enumerate}
\end{assumption}
Even though we assume that $I_1,\dots,I_p$ are explicitly given, one can compute such subsets using the procedure in~\cite{Waki06}. 
Roughly speaking, this procedure consists of two steps.
The first step provides the correlation sparsity pattern (csp) graph of the variables involved in the input polynomial data. 
The second step computes the maximal cliques of a chordal extension of this csp graph.
Even if the computation of all maximal cliques of a graph is an NP hard problem in general, it turns out that this procedure is efficient in practice, due to the properties of chordal graphs   (see, e.g.,~\cite{blair1993introduction} for more details on the properties of chordal graphs).
%
\subsection{Hankel and localizing matrices}
\label{sec:prelim_mat}
To $g \in \SymRX$ and a linear functional $L : \RX_{2 d} \to \R$, one associates the following two matrices:
\begin{enumerate}[(1)]
\item the {\em noncommutative Hankel matrix} ${\newnc{\M_d}(L)}$ is the matrix indexed by words $u,v \in \langle \underline{X} \rangle_d$, with $({\newnc{\M_d}(L)})_{u,v} = L (u^\star v)$;
\item the {\em localizing matrix} ${\newnc{\M_{d - \lceil \deg g / 2 \rceil}}(g L)}$ is the  matrix indexed by words $u,v \in \langle \underline{X} \rangle_{d - \lceil \deg g / 2 \rceil }$, with $({\newnc{\M_{d - \lceil \deg g / 2 \rceil}}(g L)})_{u,v} = L (u^\star g v)$.
\end{enumerate}
The functional $L$ is called {\em unital} if $L(1) = 1$ and is called {\em symmetric} if $L(f^\star) = L(f)$, for all $f$ belonging to the domain of $L$.
We also recall the following useful facts together with their proofs for the sake of completeness.
\begin{lemma}[\protect{\cite[Lemma~1.44]{burgdorf16}}]
\label{lemma:Hankel}
Let $g \in \SymRX$ and let $L : \RX_{2 d} \to \R$ be a symmetric linear functional. Then, one has:
\begin{enumerate}[\rm (1)]
\item $L(h^\star h) \geq 0$ for all $h \in  \RX_d$, if and only if, ${\newnc{\M_d}(L)} \succeq 0$;
\item $L(h^\star g h) \geq 0$ for all $h \in \RX_{d - \lceil \deg g / 2 \rceil }$, if and only if, ${\newnc{\M_{d - \lceil \deg g / 2 \rceil}}(g L)} \succeq 0$.
\end{enumerate} 
\end{lemma}
\begin{proof}
For $h = \sum_w h_w w \in \RX_d$, let us denote by $\h \in \R^{\bsigma(n,d)}$ the vector consisting of all coefficients $h_w$ of $h$. 
The first statement now follows from
\[
L(h^\star h) = \sum_{u,v} h_u h_v L(u^\star v) = \sum_{u,v} h_u h_v ({\newnc{\M_d}(L)})_{u,v} = \h^T {\newnc{\M_d}(L)} \h \,.
\]
The second statement follows after checking that $L(h^\star g h) = \h^T {\newnc{\M_{d - \lceil \deg g / 2 \rceil}}(g L)}  \h$.
\end{proof}
\begin{definition}
\label{def:flatextension}
Suppose $L : \RX_{2d + 2 \delta} \to \R$ is a linear functional with restriction $\tilde{L} : \RX_{2 d} \to \R$. 
We associate to $L$ and $\tilde{L}$ the Hankel matrices ${\newnc{\M_{d+ \delta}}(L)}$ and ${\newnc{\M_d}(\tilde{L})}$ respectively, and get the block form
\[
{\newnc{\M_{d + \delta}}(L)} = \begin{bmatrix}
{\newnc{\M_d}(\tilde{L})} & B \\[1mm]
B^T & C
\end{bmatrix} \,.
\]
We say that $L$ is \emph{$\delta$-flat} or that $L$ is a \emph{flat extension} of $\tilde{L}$, if ${\newnc{\M_{d+\delta}}(L)}$ is flat over ${\newnc{\M_d}(\tilde{L})}$, i.e., if $\rank {\newnc{\M_{d+\delta}}(L)} = \rank {\newnc{\M_d}(\tilde{L})} $.
\end{definition}
For a subset $I \subseteq \{1,\dots,p\}$, let us define 
${\newnc{\M_d}(L,I)}$ to be the Hankel submatrix obtained from ${\newnc{\M_d}(L)}$ after retaining only those rows and columns indexed by $w \in \langle \underline{X}(I) \rangle_d$. 
When $I \subseteq I_k$ and $g \in \RXI{k}$, for $k \in \{1,\dots,p\}$ , we define the localizing submatrix ${\newnc{\M_{d - \lceil \deg g / 2 \rceil}}(g L,I)}$ in a similar fashion.
In particular, ${\newnc{\M_d}(L,I_k)}$ and ${\newnc{\M_{d - \lceil \deg g / 2 \rceil}}(g L,I_k)}$ can be seen as Hankel and localizing matrices with rows and columns indexed by a basis of $\RXI{k}_d$ \newnc{and $\RXI{k}_{d - \lceil \deg g / 2 \rceil}$, respectively}.

\if{
The following result is a sparse variant of Lemma~\ref{lemma:Hankel}.
\begin{lemma}
\label{lemma:sparse_Hankel}
Let $g \in \SymRXI{k}$ and $L : \RX_{2 d} \to \R$ be a symmetric linear functional. Then, one has:
\begin{itemize}
\item $L(h^\star h) \geq 0$ for all $h \in \RXI{k}_d$ if and only if ${\newnc{\M}(L,I_k)} \succeq 0$;
\item $L(h^\star g h) \geq 0$ for all $h \in  \RXI{k}_{d - \lceil \deg g / 2 \rceil }$ if and only if ${\newnc{\M}(g L,I_k)} \succeq 0$.
\end{itemize} 
\end{lemma}
\begin{proof}
For $h = \sum_w h_w w \in \RXI{k}_d$, let us denote by $\h \in \R^{\bsigma(n_k,d)}$ the vector consisting of all coefficients $h_w$ of $h$. 
The first statement now follows from
\[
L(h^\star h) = \sum_{u,v} h_u h_v L(u^\star v) = \sum_{u,v} h_u h_v {\newnc{\M}(L,I_k)}_{u,v} = \h^T {\newnc{\M}(L,I_k)} \h \,.
\]
Similarly, the second statement follows after verifying that $L(h^\star g h) = \h^T {\newnc{\M}(g L,I_k)}  \h$.
\end{proof}
}\fi
\section{Sparse Representations of Noncommutative Positive Polynomials}
\label{sec:sparsePutinar}
In this section, we prove our main theoretical result, which is a sparse version of 
the Helton-McCullough archimedean Positivstellensatz (Theorem~\ref{th:densePsatz}).
For this, we rely on 
amalgamation theory for $C^\star$-algebras, see e.g.~\cite{Amalgam78,Voi83}.

Given a Hilbert space $\mathcal{H}$, we denote by $\mathcal{B}(\mathcal{H})$ the set of bounded operators on $\mathcal{H}$.
A $C^\star$-algebra is a complex Banach algebra $\mathcal{A}$ with an involution satisfying $\|xx^\star\|=\|x\|^2$
for all $x\in\mathcal A$. Equivalently, it is a 
norm
closed subalgebra with involution of $\mathcal B(\mathcal H)$ for some Hilbert space $\mathcal H$.
{
Given a $C^\star$-algebra $\cA$, \newnc{a {\em state $\varphi$} is defined to be a positive linear functional of unit norm on $\cA$}, and we write often $(\cA,\varphi)$ when $\cA$ comes together with the state $\varphi$.
Given two $C^\star$-algebras $(\cA_1,\varphi_1)$ and $(\cA_2,\varphi_2)$, a homomorphism $\iota : \cA_1 \to \cA_2$ is called  {\em state-preserving}  if $\varphi_2 \circ \iota  = \varphi_1$.
Given a $C^\star$-algebra $\cA$,  a {\em unitary representation} of $\cA$ in $\cH$ is a $*$-homomorphism $\pi : \cA \to \mathcal{B} (\cH)$ which is {\em strongly continuous}, i.e., the mapping $\cA \to \cH$, $g \mapsto \pi(g) \xi$ is continuous for every $\xi \in \cH$.
}
\begin{theorem}[\protect{\cite{Amalgam78}} or \protect{\cite[Section 5]{Voi83}}]
\label{th:amalgamation}
Let 
$(\mathcal{A},\varphi_0)$ and $\{(\mathcal B_k,\varphi_k) : k \in I\}$ be $C^\star$-algebras with states, and let $\iota_k$ be a state-preserving embedding of $\mathcal{A}$ into $\mathcal B_k$, for each $k\in I$. 
Then there exists a $C^\star$-algebra ${\mathcal D}$ amalgamating the $(\mathcal B_k,\varphi_k)$
over $(\mathcal{A},\varphi_0)$. That is,
there is a state $\varphi$ on ${\mathcal D}$, and state-preserving homomorphisms 
$j_k : \mathcal B_k \to {\mathcal D}$, such that {$j_k \circ \iota_k  = j_i \circ \iota_i$}, for all $k,i \in I$, and such that $\bigcup_{k \in I} j_k (\mathcal B_k)$ generates $D$.
\end{theorem}
{
Theorem \ref{th:amalgamation} is illustrated in Figure \ref{diag:amalgamation} in the case $I=\{1,2\}$.}
\begin{figure}[h]
\begin{center}
\begin{tikzcd}[row sep=3ex,column sep=1ex]
&  ({\mathcal D},\varphi) \\
{} \\
 \big(\mathcal{B}_1,\varphi_1\big) \arrow[uur,"j_1",dashed] &&  \big(\mathcal{B}_2,\varphi_2\big) \arrow[uul,"j_2"',dashed]  \\
{} \\
&& \\
&  \big(\mathcal{A},\varphi_0\big) \arrow[uuur, "\iota_2"] \arrow[uuul, "\iota_1"'] 
\end{tikzcd}
\caption{Illustration of Theorem~\ref{th:amalgamation} in the case $I=\{1,2\}$.}
\label{diag:amalgamation}
\end{center}
\end{figure}
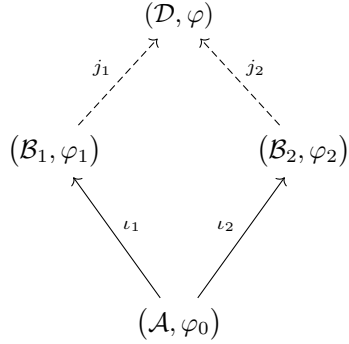


We also recall the construction by Gelfand-Naimark-Segal (GNS) establishing a correspondence between $\star$-representations of a  $C^\star$-algebra and positive linear functionals on it. 
In our context, the next result~\cite[Theorem~1.27]{burgdorf16} restricts to linear functionals on $\RX$ which are positive on an archimedean quadratic module. 
\begin{theorem}
\label{th:gns}
Let $S \subseteq \SymRX$ be given such that its quadratic module ${\cM(S)}$ is archimedean. 
Let $L : \RX \to \R$ be a nontrivial linear functional with $L({\cM(S)}) \subseteq \R^{\geq 0}$.
Then there exists a tuple $\underline{A} = (A_1,\dots,A_n) \in {D_S^\infty}$ and a vector $\v$ such that $L(f) = \langle f(\underline{A}) \v , \v \rangle$, for all $f \in \RX$.
\end{theorem}
For $k=1,\dots,p$, let us define
\begin{align*}
{\cM(S)^{k}} := \left\lbrace \sum_{i=1}^K a_i^\star s_i a_i : K \in \N \,, a_{i} \in \RXI{k} ,\, s_i \in {(S \cap \SymRXI{k}) \cup \{1\} } \right\rbrace \,,
\end{align*}
and
\begin{align}
\label{eq:sparseMS}
{\cM(S)^{\sparse} := \cM(S)^{1} + \dots + \cM(S)^{p}} \,.
\end{align}
\newnc{Next, we state the main foundational result of this paper.}
\begin{theorem}
\label{th:sparsePsatz}
Let $S \cup \{f\} \subseteq \SymRX$  and let ${D_S}$ be as in~\eqref{eq:newDS} with the additional quadratic constraints~\eqref{eq:additional}. 
Suppose Assumption~\ref{hyp:sparsityRIP} holds. 
If $f(\underline{A}) \succ 0$ for all $\underline{A} \in {D_S^\infty}$, then $f \in {\cM(S)^{\sparse}}$.
\end{theorem}
\begin{proof}
The proof is by contradiction: suppose that $f(\underline{A}) \succ 0$ for all $\underline{A} \in {D_S^\infty}$, and that $f \notin {\cM(S)^{\sparse}}$. 
By the Hahn-Banach separation theorem, also known as the Eidelheit-Kakutani Theorem in this context (see~\cite[Corollary III.1.7]{Bar02}  {or \cite[\S0.2.4]{jameson1970ordered}}), there exists a linear functional $L : \RX \to \R$ with $L(f) \leq 0$ and $L ({\cM(S)^{\sparse}}) \subseteq \R^{\geq 0}$. 
{Since $\bfone$ belongs to the algebraic interior of ${\cM(S)^{\sparse}}$ by archimedeanity and $L$ is nonzero, one has $L(\bfone) > 0$.}

Here, we cannot directly apply Theorem~\ref{th:gns} since ${\cM(S)^{\sparse}}$ is not the quadratic module of $\RX$ generated by the polynomials involved in $S$.
Nevertheless, we will prove that there exists a tuple $\underline{A} = (A_1,\dots, A_n) \in {D_S^\infty}$ and a nonzero vector $\w$ such that $L(f) = \langle f(\underline{A}) \w, \w \rangle$.
Since $f \succ 0$ implies that $\langle f(\underline{A}) \w, \w \rangle > 0$, this will contradict the fact that $L(f) \leq 0$.

For $k=1,\dots,p$, let us denote by $L^k : \RXI{k} \to \R$ the restriction of $L$ to $\RXI{k}$.
Observe that $L^k ({\cM(S)^k}) \subseteq \R^{\geq 0}$.
Each linear functional $L^k$ induces a sesquilinear form 
\[
(g,h) \mapsto \langle g, h \rangle_k := L^k(g^\star h)
\]
on $\RXI{k}$, which is positive semidefinite since $L^k$ is positive on sums of hermitian squares, allowing us to apply the Cauchy-Schwarz inequality.
Let $\mathcal{N}^k := \{ h \in \RXI{k} : \langle h,h\rangle_k = 0 \}$ be the nullvectors corresponding to $L^k$. 
By using again the Cauchy-Schwarz inequality, one can show that $\mathcal{N}^k$ is a vector subspace of $\RXI{k}$, and the sesquilinear form $L^k$ induces an inner product on the quotient space $\RXI{k} / \mathcal{N}^k$. 
Let us denote by $\mathcal{H}(I_k)$ the Hilbert space completion of $\RXI{k} / \mathcal{N}^k$ {and denote by $\langle \cdot, \cdot \rangle_k$ its inner product}.
{Since $L(\bfone) > 0$, one has $\bfone \notin \mathcal{N}^k$,} and $\mathcal{H}(I_k)$ is nontrivial and separable.
By using the fact that $L^k$ is nonnegative on the archimedean quadratic module ${\cM(S)^k}$, there exists $N \in \N$ such that $L^k(g^\star (N - X_i^2) g) \geq 0$, for all $g \in \RXI k$ {and $i \in I_k$}. 
Therefore, one has
\begin{align}
\label{eq:leftideal}
0 \leq \langle X_i g, X_i g \rangle_k = L^k(g^\star X_i^2 g) \leq N L^k(g^\star g) \,,
\end{align}
implying that $\mathcal{N}^k$ is a left ideal. 
Therefore, the left multiplication operator $\hat{X}_i^k : g \mapsto {X_i} g$ is well-defined on $\RXI{k} / \mathcal{N}^k$, for all $i \in I_k$. 
By~\eqref{eq:leftideal}, this operator is also bounded and can be extended uniquely to 
a bounded operator on
$\mathcal{H}(I_k)$. 
We fix an orthonormal basis of $\mathcal{H}(I_k)$ and denote by $\hat{A}_i^k$ the corresponding representative of the left multiplication by $X_i$ in $\mathcal{B}(\mathcal{H}(I_k))$ with respect to this basis. 
Let us denote $\hat{A}^k := (\hat{A}_i^k)_{i \in I_k}$. 
Then, one has for all $g \in \RXI{k}$
\begin{align} 
\label{eq:proofLk}
L^k(g) = \langle g(\hat{A}^k) \v^k, \v^k  \rangle_k \,,
\end{align} 
where $\v^k \in {\mathcal{H}(I_k)}$ is the image of the identity polynomial. We denote by $\varphi_k$ the state induced by $\v^k$ on $\mathcal{B}(\mathcal{H}(I_k))$, that is, {$\varphi_k( B)=\langle B\v^k,\v^k\rangle$ for $B\in\mathcal{B}(\mathcal{H}(I_k))$. In particular, $\varphi_k(g(\hat{A}^k) ) = L^k(g)$, for all $g \in \RXI{k}$.}
We refer to~Figure \ref{diag:amalgam} for an illustration for the case $p=2$. 
\newnc{Note also that given a polynomial $u  \in \RXI{k}$ with associated vector $\u\in\mathcal{H}(I_k)$, there exists a polynomial $g \in \RXI{k}$ (by construction), such that $\u = g(\hat{A}^k) \v^k$.}

Now, the proof proceeds by induction on $p$. 
With $p = 1$, this corresponds to the dense representation result stated in Theorem~\ref{th:densePsatz}.
%
%
\subsection*{Case $p=2$} 
First, note that the running intersection property~\eqref{eq:RIP} \newnc{always} holds in this case.
\if{If $I_1 \cap I_2 = \emptyset$, 
{then $f_1$ and $f_2$ involve independent variables which implies that}
 $f_1(\underline{A}) \succ 0$ and $f_2(\underline{A}) \succ 0$, for all $\underline{A} \in {D_S^\infty}$. 
As a consequence of Theorem~\ref{th:densePsatz}, one has $f_k \in {\cM(S)^k}$, which implies that $f = f_1 + f_2 \in {\cM(S)^{\sparse}}$.

Next, let us suppose that $I_1 \cap I_2 \neq \emptyset$.
}\fi
Let us define the sesquilinear form $(g,h) \mapsto \langle g, h \rangle_{12} := L^1(g^\star h)$ on $\R \langle \underline{X}( I_1 \cap I_2) \rangle$.
As above, we obtain $\mathcal{N}^{12} := \{ h \in \R \langle \underline{X}( I_1 \cap I_2) \rangle : \langle h,h\rangle_{12} = 0 \}$ and 
the Hilbert space completion $\mathcal{H}(I_1 \cap I_2)$ of $\R \langle \underline{X}( I_1 \cap I_2) \rangle / \mathcal{N}^{12}$. 
We denote by ${L^{12}}$ the restriction of $L^1$ (or, equivalently, $L^2$) to $\R \langle \underline{X}( I_1 \cap I_2) \rangle$, and by ${\varphi_{12}}$ the induced state on $\mathcal B(\mathcal{H}(I_1 \cap I_2))$.
{
Let us denote by $\hat{A}_i^{12}$ the corresponding representative of the left multiplication by $X_i$ in $\mathcal{B}(\mathcal{H}(I_1 \cap I_2))$ with respect to this basis, for $i \in I_1 \cap I_2$.
}
For $k \in \{1,2\}$, let us denote by $i_k : \R \langle \underline{X}( I_1 \cap I_2) \rangle \to \RXI{k}$ the canonical embedding.
%
%
Next we apply Theorem~\ref{th:amalgamation} with $I = \{1,2 \}$,  $\mathcal{A} = \mathcal{B} (\mathcal{H}(I_1 \cap I_2))$ endowed with ${\varphi_{12}}$, $\mathcal B_k= \mathcal{B} (\mathcal{H}(I_k))$ endowed with $\varphi_k$, and $\iota_k : \mathcal{B}(\mathcal{H}(I_1 \cap I_2)) \to \mathcal{B}(\mathcal{H}(I_k))$ being the canonical embedding, 
satisfying $\iota_k (\hat{A}_i^{12}) = \hat{A}_i^{k}$ for all $i \in I_1 \cap I_2$ \newnc{(observe that $\mathcal{B}(\mathcal{H}(I_1\cap I_2))$ contains the
algebra generated by $\hat A^{12}_i$ as a dense subset)}. 
{If $I_1 \cap I_2=\emptyset$, then we amalgamate them over $\R$, and otherwise  over $(\mathcal{B} (\mathcal{H}(I_1 \cap I_2),\varphi_{12})$.}
\newnc{Note that $\iota_k$ is state-preserving by construction.} 
As displayed in Figure~\ref{diag:amalgam}, we obtain an amalgamation {$\mathcal D$}
with state $\varphi$
and homomorphisms $j_k :\mathcal{B}(\mathcal{H}(I_k)) \to {\mathcal D}$ such that {$j_1 \circ \iota_1  = j_2 \circ \iota_2$}.

\begin{figure}
\begin{center}
\begin{tikzcd}[row sep=3ex,column sep=1ex]
&  ({\mathcal D},\varphi) \\
{} \\
 \big(\mathcal{B}(\mathcal{H}(I_1)),\varphi_1\big) \arrow[uur,"j_1",dashed] &&  \big(\mathcal{B}(\mathcal{H}(I_2)),\varphi_2\big) \arrow[uul,"j_2"',dashed]  \\
{} \\
\big(\RXI{1},L^1\big) \arrow[uu] && \big(\RXI{2},L^2\big) \arrow[uu] \\
&  \big(\mathcal{B}(\mathcal{H}(I_1\cap I_2)),{\varphi_{12}} \big) \arrow[uuur, "\iota_2"] \arrow[uuul, "\iota_1"'] \\
{} \\
&\big( \R\langle X (I_1 \cap I_2) \rangle,{L^{12}}\big) \arrow[uuur, "i_2"'] \arrow[uuul, "i_1"] \arrow[uu] \\
\end{tikzcd}
\caption{Amalgamation for the case $p=2$.}
\label{diag:amalgam}
\end{center}
\end{figure}
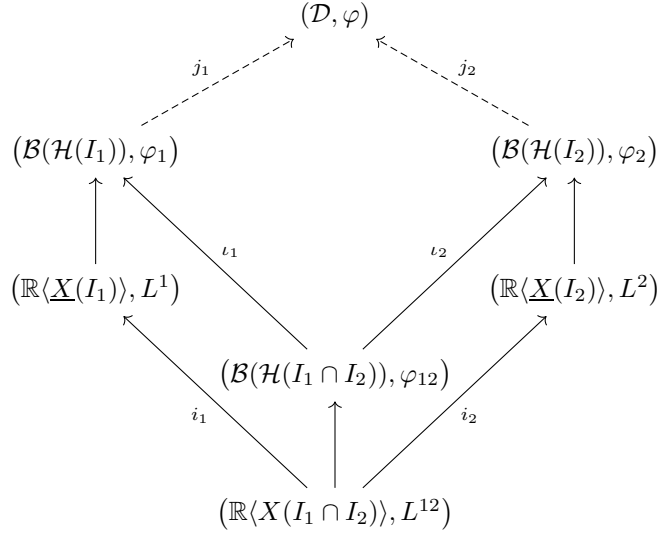

Next perform the GNS construction with $({\mathcal D},\varphi)$. 
There is a Hilbert space $\mathcal K$,
representation $\pi:{\mathcal D}\to\mathcal B(\mathcal K)$ and vector $\xi\in\mathcal K$ so that
$\varphi(a)=\langle \pi(a)\xi,\xi\rangle.$
Then, let us define $\underline{A} := (A_1,\dots,A_n)$, with
\[
A_i :=
\begin{cases} 
\pi(j_1 (\hat{A}_i^1)) &\mbox{if } i \in I_1 \,, \\
\pi(j_2 (\hat{A}_i^2)) & \mbox{if } i \in I_2 \,.
\end{cases}
\]

{By the amalgamation property, this is well-defined since $j_1(\hat{A}_i^1) = j_1 \circ \iota_1 (\hat{A}_i^{12}) = j_2 \circ \iota_2 (\hat{A}_i^{12}) = j_2(\hat{A}_i^2)$ if $i\in I_1\cap I_2$.}

For all $g \in \RX$, we now set $\tilde{L}(g) := \langle g(\underline{A}) \xi, \xi \rangle$.
We claim that $\tilde L$
extends $L^k$. 
Indeed, for $g\in\RXI k$ we have
\[
\begin{split}
\tilde L(g)&=\langle g(\underline{A}) \xi, \xi \rangle
= \langle g(\pi(j_k(\hat A^k))) \xi, \xi \rangle
= \langle \pi(g(j_k(\hat A^k))) \xi, \xi \rangle \\
& = \varphi(g(j_k(\hat A^k))) = 
\varphi(j_k(g(\hat A^k))) =
\varphi_k (g(\hat A^k)) = L^k(g).
\end{split}
\]
{The above equalities come from the fact that nc polynomials commute with homomorphisms (here $\pi, \iota_k$), since they are linear combination of products of letters and homomorphisms are addition, multiplication as well as unit (multiplicative identity) preserving.}

Therefore, 
\[\langle f(\underline{A}) \xi, \xi \rangle=\tilde{L}(f) = \tilde{L}(f_1) + \tilde{L}(f_2) = L^1(f_1) + L^2(f_2) = L(f) \leq  0.
\]

It only remains to prove that $\underline{A} \in {D_S^\infty}$, i.e., that $s(\underline{A}) \succeq 0$, for all $s \in S$. 
By assumption, $s \in \SymRXI{k}$ for some $k \in \{1,2 \}$, so 
\begin{align}
\label{eq:proofA}
s(\underline{A}) = s((\pi\circ j_k) (\hat{A}^k)) = (\pi\circ j_k) (s(\hat{A}^k)) \,.
\end{align}
Since $\RXI{k} / \mathcal{N}^k$ is dense in ${\mathcal{H}(I_k)}$, one can approximate as closely as desired $\u \in {\mathcal{H}(I_k)} $ by elements of $\RXI{k} / \mathcal{N}^k$.
We prove that $\langle s(\hat{A}^k) \u, \u \rangle_k \geq 0$, where $\u$ stands for a vector representative of $u \in \RXI{k} / \mathcal{N}^k$.
Given such a vector $\u$, there exists a polynomial $g \in \RXI{k}$, such that $\u = g(\hat{A}^k) \v^k$. 
Next, the following holds:
\[
\langle s(\hat{A}^k) \u, \u \rangle_k = 
\langle s(\hat{A}^k) g(\hat{A}^k) \v^k, g(\hat{A}^k) \v^k \rangle_k = 
\langle (s g) (\hat{A}^k)  \v^k, g(\hat{A}^k) \v^k \rangle_k = L^k(g^\star s g) \,,
\]
where the last equality comes from~\eqref{eq:proofLk}.
Since $g^\star s g \in {\cM(S)^k}$ and $L^k({\cM(S)^k}) \subseteq \R^{\geq 0}$, one has $\langle s(\hat{A}^k) \u, \u \rangle_k \geq 0$, which implies that $c := s(\hat{A}^k) \succeq 0$. 
Since $c$ is a nonnegative element of the $C^\star$-algebra $ \mathcal{B}(\mathcal{H}(I_k))$, there exists $b \in \mathcal{B}( \mathcal{H}(I_k))$ such that $c = b^\star b$. 
Eventually by~\eqref{eq:proofA}, one has 
$s(\underline{A}) = (\pi \circ j_k)(s(\hat{A}^k)) = \pi ( j_k (c) )= \pi ( j_k (b^\star b)) =  \pi (j_k (b))^\star \pi ( j_k(b)) \succeq 0 $, yielding $\underline{A} \in {D_S^\infty}$, the desired result.
%
\subsection*{General case}
Now assume $p>2$. 
{
For each $m\leq p$ we will construct a Hilbert space $\mathcal H (\cup_{j\leq m} I_j )$ with state $\tilde \varphi_m$  acting on $\mathcal B (\mathcal H(\cup_{j\leq m} I_j  )$, a tuple ${\underline{A}^m\in D_S^\infty}$ as well as a unit vector $\xi^m $ such that the linear functional $\tilde L^m : \RX \to \R$, defined by 
\begin{equation}
\label{eq:newL}
\tilde L^m(g):=\langle g({\underline{A}^m})\xi^m,\xi^m\rangle \,,
\end{equation}
extends $L_j$ for each $j \leq m$, implying that $\tilde L^m(g) = L(g)$ 
for all $g\in\sum_{j\leq m} \R\langle \underline X(I_j)\rangle$. 

The basis for the induction, $p\leq2$, has been established above. \\
Let $p>m\geq2$ and assume by induction that we have $\mathcal H (\cup_{j\leq m} I_j )$, $\tilde \varphi_m$, ${\underline{A}^m\in D_S^\infty}$ and $\tilde L^m$ as above.
By the running intersection property~\eqref{eq:RIP}, there is $k\leq m$ with $\big( \cup_{j\leq m} I_j \big)\cap I_{m+1}\subseteq I_k$.
Recall that $L^{m+1}$ is defined as the restriction of $L$ to $\RXI{m+1}$. 
Let $L^0$ be the restriction of $L$ (or, equivalently, of $\tilde L^m$) to $\RXI k$. As before, 
Theorem \ref{th:gns} produces Hilbert spaces {$\mathcal H(I_{m+1}),\mathcal H_0 := \mathcal H( \big(\cup_{j\leq m} I_j\big)\cap I_{m+1}   )$, operators $\hat A^{m+1}, \hat A^{0}$} and  states $\varphi_{m+1},\varphi_0$ acting on $\mathcal B(\mathcal H (I_{m+1})),\mathcal B(\mathcal H_0)$. 

The operator $\hat{A}^{m+1}$ and  state $\varphi_{m+1}$ satisfy $\varphi_{m+1}( g(\hat{A}^{m+1}) ) = L^{m+1}(g)$, for all $g \in \RXI{m+1}$. 
In addition, the canonical embeddings $ \iota : \mathcal B(\mathcal H_0) \to \mathcal B(\mathcal H (\cup_{j\leq m} I_j ))$, $ \iota_{m+1} : \mathcal B(\mathcal H_0) \to \mathcal B(\mathcal H ( I_{m+1} ))$, and the operator $\hat A^0$ satisfy $ \iota (\hat A_i^0) =  A_i^m$ and $ \iota_{m+1} (\hat A_i^0) = \hat  A_i^{m+1}$, for all $i \in ( \cup_{j\leq m} I_j \big)\cap I_{m+1} $.

The remaining part of the proof is very similar to the case $p = 2$. 
We amalgamate $\mathcal B(\mathcal H (\cup_{j\leq m} I_j ))$ and $\mathcal B(\mathcal H(I_{m+1}))$; if $\big(\cup_{j\leq m} I_j\big)\cap I_{m+1}=\emptyset$, then we amalgamate them over $\R$, and otherwise  over
$(\mathcal B(\mathcal H_0),\varphi_0)$.
Doing so, we obtain an amalgamation $\mathcal{D}_{m+1}$ and two homomorphisms $j : \mathcal B(\mathcal H (\cup_{j\leq m} I_j )) \to \mathcal{D}_{m+1}$ and $ j_{m+1} : \mathcal B(\mathcal H (I_{m+1} )) \to \mathcal{D}_{m+1}$.
%
Applying the GNS construction
to the amalgamated $C^\star$-algebra 
then yields a Hilbert space $\mathcal K_{m+1}$, a 
representation $\pi_{m+1}:{\mathcal D_{m+1}}\to\mathcal B(\mathcal K_{m+1})$, a unit vector $\xi^{m+1}\in\mathcal K_{m+1}$ and we can define $\underline{A}^{m+1}$ with
\[
A_i^{m+1} :=
\begin{cases} 
\pi_{m+1}(j (A_i^m)) &\mbox{if } i \in \cup_{j\leq m} I_j \,, \\
\pi_{m+1}(j_{m+1} (\hat{A}_i^{m+1})) & \mbox{if } i \in I_{m+1} \,,
\end{cases}
\]
as well as $\tilde L^{m+1}(g) := \langle g({\underline{A}^{m+1}}) {\xi^{m+1}}, {\xi^{m+1}} \rangle$.

As in the case $p=2$, by the amalgamation property, $\underline{A}^{m+1}$ is well-defined since $j(A_i^m) = j \circ \iota (\hat{A}_i^{0}) = j_{m+1} \circ \iota_{m+1} (\hat{A}_i^{0}) = j_{m+1}(\hat{A}_i^{m+1})$ if $i\in \big( \cup_{j\leq m} I_j \big)\cap I_{m+1}$. One proves as before that $\underline{A}^{m+1} \in D_S^\infty$. 
In addition, $\tilde L^{m+1}(g) = \tilde L^m(g) = L(g)$  for all $g\in\sum_{j\leq m} \R\langle \underline X(I_j)\rangle$, where the first equality comes from the definition of $\underline{A}^{m+1}$ and the second one comes from the induction hypothesis.
One has $\tilde L^{m+1} = L^{m+1} = L(g)$ for all $g \in \RXI{m+1}$, which implies that $L(g)=\langle g({\underline{A}^{m+1}}) {\xi^{m+1}}, {\xi^{m+1}} \rangle$
for all
$g\in\sum_{j\leq m+1} \R\langle \underline X(I_j)\rangle$. 

For $m=p$, we obtain $\underline{A}^{p} \in D_S^\infty$ and a unit vector $\xi^p$ such that $\langle f({\underline{A}^{p}}) {\xi^{p}}, {\xi^{p}} \rangle = L(f) \leq  0$, yielding the desired conclusion.
}
\end{proof}
\if{
{
\begin{remark}
Theorem~\ref{th:sparsePsatz} can be seen as a noncommutative variant of the result by Lasserre stated in~\cite[Corollar~3.9]{Las06}, related to the existance of a sparse Putinar's representation in the context of sparse polynomial optimization. 
In the sparse commutative case, Lasserre builds a probability measure supported on the set of constraints, which is constructed from the  after assuming that the
Our proof is similar as it builds a linear 
\end{remark}
}
}\fi
The reader will notice that the RIP property is used subtly in the proof of Theorem~\ref{th:sparsePsatz}.
Next, we provide an example demonstrating that sparsity without a RIP-type condition is not sufficient to deduce sparsity in SOHS decompositions.
\begin{example}\rm
Consider the case of three variables $\underline X=(X_1,X_2,X_3)$
and the polynomial
\begin{multline*}
f=(X_1+X_2+X_3)^2 \\ = X_1^2+X_2^2+X_3^2+X_1X_2+X_2X_1+X_1X_3+X_3X_1+X_2X_3+X_3X_2 \in\Sigma\langle\underline X\rangle.
\end{multline*}
Then $f=f_1+f_2+f_3$, with
\[
\begin{split}
f_1 & =\frac12 X_1^2+\frac12 X_2^2+X_1X_2+X_2X_1\in\R\langle X_1,X_2\rangle, \\
f_2 & = \frac12X_2^2+\frac12X_3^2+X_2X_3+X_3X_2\in\R\langle X_2,X_3\rangle,\\
f_3 & = \frac12X_1^2+ {\frac12} X_3^2+X_1X_3+X_3X_1\in\R\langle X_1,X_3\rangle.
\end{split}
\]
However,
the sets $I_1=\{1,2\}$, $I_2=\{2,3\}$ and  $I_3=\{1,3\}$ do not satisfy the RIP condition~\eqref{eq:RIP} and 
 $f\not\in\SigmaX^{\sparse}:=\Sigma\langle X_1,X_2\rangle+\Sigma\langle X_2,X_3\rangle+\Sigma\langle X_1,X_3\rangle$ since it has a unique Gram matrix by homogeneity.

Now consider $S=\{1-X_1^2,\, 1-X_2^2,\, 1-X_3^2\}$. 
Then ${D_S}$ is as in~\eqref{eq:newDS}, ${\cM(S)^{\sparse}}$ is as in~\eqref{eq:sparseMS} and $f|_{{D_S^\infty}}\succeq0$. 
However, we claim that 
{$f-\lambda\in {\cM(S)^{\sparse}}$ iff $\lambda\leq -3$. }
Clearly, 
 \[
f+3=(X_1+X_2)^2+(X_1+X_3)^2+(X_2+X_3)^2+(1-X_1^2)+(1-X_2^2)+(1-X_3^2)\in {\cM(S)^{\sparse}}.
 \]
{
So one has $-3  \leq \sup \{\lambda : f - \lambda \in \cM(S)^{\sparse} \}$, and the dual of this latter problem is given by
\begin{equation}
\label{eq:norip}
\begin{aligned}
\inf_{L_k} \quad  &  \sum_{k=1}^3 L_k(f_k)  \\	
\text{s.t.} 
\quad & L_k(1) = 1 \,, \quad k=1,\dots,3 \,,  \\
\quad & L_k(h^\star h) \succeq 0 \quad \forall h \in \RXI{k} \,,  \quad k=1,\dots,3\,, \\
\quad & L_k(h^\star (1 - X_k^2) h)  \succeq 0 \quad \forall h \in \RXI{k} \,, \quad k=1,\dots,3 \,,\\
\quad &  L_j|_{\R\langle \underline X(I_j\cap I_k)\rangle}=L_k|_{\R\langle \underline X(I_j\cap I_k)\rangle} \,, \quad j,k=1,\dots,3 \,. \\ 
\end{aligned}
\end{equation}
Hence, by weak duality, it suffices to show that there exist linear functionals
$L_k:\RXI k\to\R$ satisfying the constraints of problem \eqref{eq:norip} and such that $\sum_k L_k(f_k)=-3$.
}
\if{
\begin{enumerate}[\rm (1)]
\item\label{it:prop1}
$L_k(1)=1$;
\item\label{it:prop2}
{The corresponding (infinite) Hankel matrix  $\newnc{\M}(L, I_k)$ and its
localizing variants $\newnc{\M}(g L,I_k)$ are positive semidefinite for all $g \in S$ such that $g \in \RXI{k}$};
\item\label{it:prop3}
$L_j|_{\R\langle \underline X(I_j\cap I_k)\rangle}=L_k|_{\R\langle \underline X(I_j\cap I_k)\rangle}$ for all $j,k$;\\[-3mm]
\item\label{it:prop4}
$\sum_k L_k(f_k)=-3$.
\end{enumerate}
}\fi
Define
\[
A=\begin{bmatrix}0&1\\ 1&0\end{bmatrix}, \quad B=-A
\]
and let 
\[L_k(g)=\Trace g(A,B) \quad\text{ for }g\in\RXI k.\]
{
Since $L_k(f_k)=-1$, the three first constraints of problem \eqref{eq:norip} are easily verified and $\sum_k L_k(f_k)=-3$.
 For the last one, given, say 
$h\in\RXI 1\cap\RXI 2=\R\langle X_2\rangle$, we have 
\[
\begin{split}
{L_1(h)}&=\Trace h(B), \\
L_2(h)&=\Trace h(A),
\end{split}
\]
since $L_1$ (resp.~$L_2$) is defined on  $\R\langle X_1, X_2\rangle$ (resp.~$\R\langle X_2, X_3\rangle$) and $h$ depends only on the second (resp.~first) variable $X_2$ corresponding to $B$ (resp.~$A$).}

But matrices $A$ and $B$ are orthogonally equivalent as
$UAU^T=B$ for
\[
U=\left[
\begin{array}{rr}
 0 & 1 \\
 -1 & 0 \\
\end{array}
\right],
\]
whence $h(B)=h(UAU^T)=Uh(A)U^T$ and $h(A)$ have the same trace.
 \end{example}
\section{Sparse GNS Construction and Optimizer Extraction}
\label{sec:extract}
The aim of this section is to provide a general algorithm to extract solutions of sparse noncommutative optimization problems.
We will apply this algorithm below to eigenvalue optimization (Section~\ref{sec:eig}) and trace optimization (Section~\ref{sec:trace}).
For this purpose, we first present sparse noncommutative versions of theorems by Curto and Fialkow.
In the commutative case, Curto and Fialkow provided sufficient conditions for linear functionals on the set of degree $2 d$ polynomials to be represented  by integration with respect to a nonnegative measure.
The main sufficient condition to guarantee such a representation is flatness (see  Definition~\ref{def:flatextension}) of the corresponding Hankel matrix. {This notion was exploited in a noncommutative  setting for the first time by
McCullough \cite{McCullSOS} in his proof of the Helton-McCullough Sums of Squares theorem, cf.~\cite[Lemma 2.2]{McCullSOS}.}

In the dense case {\cite{pironio2010convergent} (see also ~\cite[Chapter 21]{anjos2011handbook} and~\cite[Theorem~1.69]{burgdorf16}) provides a first noncommutative variant for the eigenvalue problem. See~\cite{nctrace} for a similar construction for the trace problem.}
As this will be needed in the sequel, we recall this theorem and a sketch of its proof, which relies on a finite-dimensional GNS construction. 
\begin{theorem}
\label{th:dense_flat}
Let $S  \subseteq \SymRX$ and set $\delta := \max \{ \lceil \deg (g)/2 \rceil : g \in S\cup {\{1\}} \}$.
{For $d \in \N$,} let $L : \RX_{2 d + 2 \delta} \to \R$ be a unital linear functional satisfying $L({\cM(S)_{d +  \delta}}) \subseteq \R^{\geq 0}$.
If $L$ is  $\delta$-flat, then there exist $\hat{A} \in {D_S^r}$ for some $r \leq \bsigma(n,d)$ and a unit vector $\v$ such that
\begin{align}
\label{eq:flat_representation}
{L(g) = \langle g(\hat{A}) \v , \v \rangle } \,,
\end{align}
{for all $g \in \SymRX_{2d}$}.
\end{theorem}
\begin{proof}
Let $r := \rank {\newnc{\M_{d+\delta}}(L)}$.
Since ${\newnc{\M_{d+\delta}}(L)}$ is a positive semidefinite matrix, we obtain the Gram matrix decomposition ${\newnc{\M_{d+\delta}}(L)} = [\langle \u , \w \rangle]_{u,w}$ with vectors $\u,\w \in \R^r$, where the labels are words of degree at most $d + \delta$. 
Then, we define the following finite-dimensional Hilbert space  
\[
\cH := \text{span} \, \{\w \mid \deg w \leq d + \delta \} = \text{span} \, \{\w \mid \deg w \leq d \} ,
\]
where the equality comes from the flatness assumption.
Afterwards, one can directly consider the operators $\hat{A}_i$ representing the left multiplication by $X_i$ on $\cH$, i.e., $\hat{A}_i \w = \X_i \w$.
Thanks to the flatness assumption, the operators $\hat{A}_i$ are well-defined and one can show that they are symmetric. 
{
Let $\hat{A} := (\hat{A}_1,\dots,\hat{A}_n)$.
As in the GNS construction of Theorem~\ref{th:sparsePsatz}, one has $L(g) = \langle g(\hat{A}) \v , \v \rangle$, with $\v$ being the vector representing $\bfone$ in $\cH$.
Given $s \in S$, let us  prove that $\langle s(\hat{A}) \w , \w \rangle \geq 0$, for all $\w \in \mathcal{H}$. 
By construction, any vector $\w \in \mathcal{H}$ can be written as $g(\hat{A}) \v$, for some polynomial $g\in \SymRX_{2d}$. 
Thus, one has $\langle s(\hat{A}) \w , \w \rangle = \langle s(\hat{A}) g(\hat{A}) \v , g(\hat{A}) \v \rangle = L(g^\star s g) \geq 0$ since $g^\star s g \in{\cM(S)_{d +  \delta}} $.
Thus, one has $\hat{A} \in {D_S^r}$, the desired result.
}
\end{proof}
We now give the sparse version of Theorem~\ref{th:dense_flat}.

\begin{theorem}
\label{th:sparse_flat}
Let $S  \subseteq \SymRX_{2d}$,  and assume ${D_S}$ is as in~\eqref{eq:newDS} with the additional quadratic constraints~\eqref{eq:additional}. 
Suppose Assumption~\ref{hyp:sparsityRIP}(i) holds.
Set $\delta := \max \{ \lceil \deg (g)/2 \rceil : g \in S\cup {\{1\}} \}$.
Let $L : \RX_{2 d + 2 \delta} \to \R$ be a unital linear functional satisfying $L({\cM(S)_d^{\sparse}}) \subseteq \R^{\geq 0}$.
Assume that the following holds:
\begin{enumerate}
\item[(H1)] ${\newnc{\M_{d+\delta}}(L,I_k)}$ and ${\newnc{\M_{d+\delta}}(L,I_k\cap I_j)}$ are $\delta$-flat, for all $j,k \in \{1,\dots,p\}$.
\end{enumerate}
Then, there exist finite-dimensional Hilbert spaces $\cH(I_k)$ with dimension $r_k$, for all $k \in \{1,\dots,p \}$, 
Hilbert spaces
$\cH(I_j \cap I_k)\subseteq\cH(I_j),\cH(I_k)$ for all pairs $(j,k)$ with $I_j \cap I_k \neq 0$, and operators $\hat A^k$, $\hat A^{j k}$,  acting on them, respectively.
Further, there are unit vectors $\v^j\in\cH(I_j)$ and $\v^{jk}\in\cH(I_j\cap I_k)$ such that
\begin{equation}\label{eq:desired}
\begin{split}
L(f) & = \langle f(\hat{A^j}) \v^j , \v^j \rangle \quad \text{for all } f\in\RXI j_{2d},\\
L(g) & = \langle g(\hat{A^{jk}}) \v^{jk} , \v^{jk} \rangle  \quad \text{for all } g\in\R \langle \underline{X}( I_j \cap I_k) \rangle_{2d}.
\end{split}
\end{equation}

Assuming that for all pairs $(j,k)$ with $I_j \cap I_k \neq {\emptyset}$, one has
\begin{enumerate}
\item[(H2)] the matrices $(\hat A_i^{j k})_{i \in I_j \cap I_k}$ have no common complex invariant subspaces,
\end{enumerate}
 then there exist $\underline{A} \in {D_S^r}$, with $r := r_1 \cdots r_p$, and a unit vector $\v$ such that
\begin{align}
\label{eq:sparse_flat_representation}
L(f) = \langle f(\underline{A}) \v , \v \rangle  \,,
\end{align}
for all $f\in\sum_j \RXI j_{2d}$.
\end{theorem}

In the proof of Theorem \ref{th:sparse_flat}
we will make use of the following simple linear algebra observation.

\begin{lemma}\label{lem:linalg1}
Let $Z\in M_n(\R)$. If
$\Trace(ZA)=0$  for all $A\in M_n(\R)$, then $Z=0$.
\end{lemma}

\begin{proof}
We have $\Trace(ZZ^T)=0$ whence $ZZ^T=0$ and thus
$Z=0$.
\end{proof}

\begin{proof}[Proof of Theorem \ref{th:sparse_flat}]
Start by applying Theorem \ref{th:dense_flat} to
$L|_{\RXI j}$ and $L|_{\R\langle \underline X(I_j\cap I_k)\rangle}$ to obtain the desired {(real)} Hilbert spaces
$\cH(I_j)$, $\cH(I_j\cap I_k)$, unit vectors $\v ^j$, $\v^{jk}$ 
and operators $\hat A^j,$ $\hat A^{jk}$ satisfying
\eqref{eq:desired}. Note that we may assume $\cH(I_j\cap I_k)\subseteq\cH(I_j),\cH(I_k)$ as the map $f(\hat A^{jk})\v^{jk}\mapsto f(\hat A^{j})\v^j$ is an isometry by construction. Then 
\begin{equation}\label{eq:restrict}
\hat A^{jk}=\hat A^j|_{\cH(I_j\cap I_k)}=\hat A^k|_{\cH(I_j\cap I_k)}.
\end{equation}

Let us denote by $\cA(I_j)$ and $\cA(I_j\cap I_k)$ the algebras generated by $\hat A^j$,  and $\hat A^{jk}$, respectively.
By \eqref{eq:restrict}, the map $\hat A_\ell^{jk}\mapsto \hat A_\ell^j$ is a $\star$-homomorphism $\cA(I_j)\to\cA(I_j\cap I_k)$.
With $r_k=\dim \cH(I_k)$ and $r_{jk}=\dim \cH(I_j\cap I_k)$, one has $\cA (I_k) \subseteq \Mbb_{r_k}(\R)$ and $\cA(I_j\cap I_k)\subseteq \Mbb_{r_{jk}}(\R)$.
%

We next want to find a finite-dimensional $C^\star$-algebra $\cA$, i.e., a subalgebra of $\Mbb_m(\R)$ for some $m\in\N$, making the diagram in Figure~\ref{diag:amalgamSmall} commute.

{
In the sequel, the proof proceeds by induction on $p$ and we focus specifically on the case $p = 2$, as the general case then follows by a simple inductive argument.
}
By the amalgamation property of $C^\star$-algebras stated in Theorem \ref{th:amalgamation}, we can always find such an \emph{infinite-dimensional} $\cA$.
However, as shown in Example~\ref{ex:no_finite_amalgam}, there may not be a suitable finite-dimensional $\cA$.
\begin{figure}[ht]
\begin{center}
\begin{tikzcd}[row sep=3ex,column sep=1ex]
& \mathcal A \\
{} \\
\mathcal A(I_1) \arrow[uur,"j_1",dashed] && \mathcal A(I_2) \arrow[uul,"j_2"',dashed]  \\
{} \\
& \mathcal A(I_1\cap I_2) \arrow[uur, "\iota_2"'] \arrow[uul, "\iota_1"] 
\end{tikzcd}
\caption{Amalgamation of finite-dimensional $C^\star$-algebras}\label{diag:amalgamSmall}
\label{fig:2}
\end{center}
\end{figure}
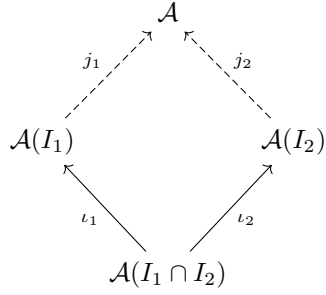
To ensure this, we assume that (H2) holds, namely that the matrices $(\hat A_i^{1 2})_{i \in I_1 \cap I_2}$ have no common complex invariant subspaces, {which implies by Burnside's theorem (see, e.g., \cite[Corollary 5.23]{bresar2014}) that} $\cA (I_1\cap I_2)= \Mbb_{r_{12}}(\R)$.
Then, for all $A\in \cA (I_1\cap I_2) = \Mbb_{r_{1 2}}(\R)$, $\iota_k(A)$ is just a direct sum of copies of $A$, up to orthogonal equivalence (by the Skolem-Noether theorem~\cite[Section III.3]{BOns}), 
i.e., there are orthogonal matrices $U_k$ such that
\[
\iota_k(A)=U_k^{T}( \I_{r_k / r_{12}}\otimes A) U_k \,,
\]
for all $k \in \{1,2\}$.
By replacing $\hat A^{k}$ with their conjugates 
$U_k^T\hat A^kU_k$,
and $\v^k$ by $U_k^T\v^k$, we may without loss of generality assume $\iota_k(A)= \I_{r_k / r_{12}}\otimes A$.

The linear functional $L$ induces linear functionals
$\check L^k$, $\check L^{12}$ on $\cA(I_k)$, $\cA(I_1\cap I_2)$ given by
$B\mapsto \Trace(B\v^k(\v^k)^T)$
and
$C\mapsto \Trace(C\v^{12}(\v^{12})^T)$, respectively.
 Write $\v^k=\sum_{j=1}^{r_k/r_{12}} e_j^k \otimes u^k_j$ for the standard basis vectors $e_j^k\in\R^{r_k/r_{12}}$ and some vectors $u^k_j\in\R^{r_{12}}$. 
Then for $C\in \mathcal A(I_1\cap I_2)= M_{r_{12}}(\R)$
 we have
 \[
 \begin{split}
\check L^{12}(C) &= \Trace(C \v^{12}(\v^{12})^T) \\
& = \check L^k( \I_{r_k/r_{12}}\otimes C) = 
\Trace((\I \otimes C)\v^k(\v^k)^T)\\
&= \Trace\Big((\I\otimes C)(\sum_{j=1}^{r_k/r_{12}}e_j^k\otimes u^k_j)(\sum_{j=1}^{r_k/r_{12}}e_j^k\otimes u^k_j)^T\Big)\\
&= \Trace\Big((\I\otimes C)\sum_{i,j=1}^{r_k/r_{12}}(e_j^k (e_i^k)^T)\otimes (u^k_j(u^k_i)^T)\Big)\\
&=  \sum_{i,j=1}^{r_k/r_{12}}\Trace\big((e_j^k (e_i^k)^T)\otimes Cu^k_j(u^k_i)^T\big)   \\
&=  \sum_{i,j=1}^{r_k/r_{12}}\Trace\big(e_j^k (e_i^k)^T\,\big)\Trace\big( Cu^k_j(u^k_i)^T\big)   \\
&= \Trace\big( C \sum_{j=1}^{r_k/r_{12}}u^k_j(u^k_j)^T)\big).
\end{split}
\]
From the equality $\Trace(C \v^{12}(\v^{12})^T)=\Trace\big( C \sum_{j=1}^{r_k/r_{12}}u^k_j(u^k_j)^T)\big)$ for all $C\in M_{r_{12}}(\R)$ we deduce using Lemma \ref{lem:linalg1}
that $\v^{12}(\v^{12})^T=\sum_{j=1}^{r_k/r_{12}}u^k_j(u^k_j)^T$. Since the left-hand side outer product is rank one, each of the $u_j^k$ must be a scalar multiple of $\v^{12}$, say $u_j^k=\lambda_{jk}\v^{12}$.
Thus $\v^k=\sum_{j=1}^{r_k/r_{12}}\lambda_{jk}e_j^k\otimes\v^{12}$
and $\sum_j \lambda_{jk}^2=1$ since $\|\v^k\|=1$.

Now set $\cA := \Mbb_{r_1 r_2}(\R)$ and define $j_1 (A) : =\I_{r_2}\otimes A$, for all $A \in \cA(I_1)$, and 
\begin{equation}\label{eq:j2}
j_2(B) := (U^T\otimes \I_{r_{12}})( \I_{r_1}\otimes B)(U\otimes \I_{r_{12}}) \,,
\end{equation}
for all $B \in \cA(I_2)$. Here $U$ is an $r_1 r_2/r_{12}$ orthogonal matrix to be determined later. 
This amalgamates the diagram in Figure~\ref{diag:amalgamSmall} (independently of the choice of $U$).

Each extension of the linear functional $\check L^k$ to
a linear functional on $\cA$
 is of the form
\begin{equation}\label{eq:extendingPain}
C\mapsto \Trace\big(C \sum_{\ell=1}^{r_{3-k}}\mu_{\ell k} e_{\ell}^{3-k}\otimes \v^k\big)=
\Trace\Big(C \underbrace{\sum_{\ell}\sum_{j=1}^{r_k/r_{12}} \mu_{\ell k}e_{\ell}^{3-k}\otimes \lambda_{jk}e_j^k}_{\w^k}\otimes\v^{12}\Big),
\end{equation}
where $\sum_\ell \mu_{\ell k}^2=1$.
Since the vectors $\w^k$ are norm one, there is a unitary $U$ with $U\w_1=\w_2$. Using this $U$ in the definition \eqref{eq:j2}, the extension \eqref{eq:extendingPain} of $\check L^1$ to a linear functional $\check L:\mathcal A\to\R$ also extends $\check L^2$ (via $j_2$).

Now define the operators $\underline{A} := (A_1,\dots,A_n)$, with
\[
A_i :=
\begin{cases} 
j_1 (\hat A_i^1) &\mbox{if } i \in I_1 \,, \\
j_2 (\hat A_i^2) & \mbox{if } i \in I_2 \,.
\end{cases}
\]
Then $L(f)=\langle f(\underline A) (\w^1\otimes \v^{12}),\w^1\otimes\v^{12}\rangle$ for all $f\in\RXI 1_{2d}+\RXI 2_{2d}$.

To conclude the proof note that each $A_i$ is symmetric  and that $\underline A\in {D_S^{r_1 r_2}}$. 
For the latter we use the fact that each constraint $g$ is either in $\RXI 1$ or $\RXI 2$, and that $\star$-subalgebras of matrix algebras admit square roots of positive semidefinite operators.
\end{proof}
\begin{example}[Non-amalgamation in the category of finite-dimensional algebras]
\label{ex:no_finite_amalgam}\rm
For given $I_1,I_2$, suppose $\cA(I_1\cap I_2)$ is generated by the $2\times2$ diagonal matrix
\[A^{1 2} =\begin{pmatrix}
1 \\ &2
\end{pmatrix},
\]
and assume $\cA(I_1)=\cA(I_2)= \Mbb_3(\R)$. 
(Observe that $\cA(I_1\cap I_2)$ is the algebra of all diagonal matrices.)
For each $k \in \{1,2\}$, let us define $\iota_k(A):= A \oplus k$, for all $A \in \cA(I_1\cap I_2)$.
We claim that there is no finite-dimensional
$C^\star$-algebra
 $\cA$ amalgamating the above Figure \ref{fig:2}.
Indeed, by the Skolem-Noether theorem, every homomorphism $\Mbb_n(\R)\to \Mbb_m(\R)$ is of the form
$x\mapsto P^{-1}(x\otimes \I_{m/n})P$ for some invertible $P$; in particular, $n$ divides $m$. 
If a desired $\cA$ existed, then the matrices
$(A^{12}\oplus 1)\otimes \I_k$ and $(A^{12}\oplus 2)\otimes \I_k$ would be similar.
But they are not as is easily seen from eigenvalue multiplicities.
\end{example}
\begin{remark}
\label{rk:gns_sparse}\rm
Theorem~\ref{th:sparse_flat} can be seen as a noncommutative variant of the result by Lasserre stated in~\cite[Theorem~3.7]{Las06}, related to the minimizers extraction in the context of sparse polynomial optimization. 
In the sparse commutative case, Lasserre assumes flatness of each moment matrix indexed in the canonical basis of $\R[\uX(I_k)]_d$, for each $k \in \{1,\dots,p\}$, which is similar to our flatness condition~(H1). 
{
The difference is that this technical flatness condition on each $I_k$ adapts to the degree of the constraints polynomials on variables in $I_k$, resulting in an adapted parameter $\delta_k$ instead of global $\delta$. 
We could assume the same in Theorem~\ref{th:sparse_flat} but for the sake of simplicity, we assume that these parameters are all equal. 
}
In addition, Lasserre assumes that each moment matrix indexed in the canonical basis of $\R[\uX(I_j \cap I_k)]_d$ is rank one, for all pairs $(j,k)$ with $I_j \cap I_k \neq \emptyset$, which is the commutative analog of our irreducibility condition~(H2).
\end{remark}
\subsection{Implementing the Sparse GNS Construction}

As in the dense case, we can summarize the sparse GNS construction procedure described in the proof of Theorem~\ref{th:sparse_flat} into an algorithm, called $\sparsegns$, stated below in Algorithm~\ref{algorithm:sparse_gns}, for the case $p=2$ (the general case is similar).

This algorithm describes how to compute the tuple $\underline{A} = (A_1,\dots,A_n)$ of amalgamated matrices acting on $\cH=\R^{r_1 r_2}\cong\cH(I_1)\otimes \R^{r_2}\cong\R^{r_1}\otimes\cH(I_2)$, and a vector $\v$ satisfying \eqref{eq:sparse_flat_representation}. 
\newnc{To check the irreducibility (H2) condition, in Line \ref{en:checkH1H2} of the algorithm, one relies on the 
Burnside theorem from matrix theory (see, e.g., \cite[Corollary 5.23]{bresar2014}): the algebra generated by $e \times e$ (real) symmetric matrices is irreducible if and only if it is isomorphic to $\Mbb_e(\R)$. So one only needs to check the dimension of the algebra.
}

\begin{algorithm}
$\sparsegns$
\label{algorithm:sparse_gns}
\begin{algorithmic}[1]
\Require ${\newnc{\M_d}(L)}$, Hankel matrix of $L$. 
\State \label{en:1}
Apply the GNS construction to obtain $\cH(I_1)$, $\cH(I_2)$ and $\cH(I_1 \cap I_2)$ of respective dimensions $r_1$, $r_2$ and $r_{1 2}$, associated to ${\newnc{\M_d}(L,I_1)}$, ${\newnc{\M_d}(L,I_2)}$ and ${\newnc{\M_d}(L,I_1 \cap I_2)}$, as well as $\hat A^1$, $\hat A^2$ and $\hat A^{1 2}$ acting on $\cH(I_1)$, $\cH(I_2)$ and $\cH(I_1 \cap I_2)$, respectively. \Comment{~the dense GNS algorithm is implemented in e.g.~$\ncsostools$~\cite{cafuta2011ncsostools}}
\State Find the corresponding unit vectors $\v^1\in \cH(I_1)$, $\v^2\in \cH(I_2)$ and $\v^{12}\in\cH(I_1 \cap I_2)$ so that \eqref{eq:desired} holds.
\If{The flatness (H1) and irreducibility (H2) conditions from Theorem~\ref{th:sparse_flat} do not hold}  \label{en:checkH1H2} 
\State Stop
\EndIf 
\For{$k \in \{1,2\}$, $i\in I_1\cap I_2$}
\State \label{en:2} \newnc{Compute $(\chi_{i,\ell}^k)_\ell$ such that the block diagonalization $\hat A_i^k = \op{diag}(\chi_{i,\ell}^k)_\ell$ holds}.
\Comment{e.g., by~\cite[Algorithm 4.1]{MKKK10}}
\State \label{en:3}
Compute invertible matrices $(P_\ell)_{\ell>1}$ such that 
$
P_{\ell}^{-1} \chi_{i,\ell}^k P_{\ell} =\chi_{i,1}^k
$
\State \label{en:3b}
Normalize each $P_{\ell}$ to make it orthogonal.
Use them to change the basis in the blocks $(\chi_{i,\ell}^k)_{\ell>1}$ \Comment{Thus, one has $\hat A_i^k= \I \otimes\chi_{i}^k$}
\State \label{en:4}
Compute an orthogonal $P$ such that
$P^{-1}\chi_i^k P = \hat A_i^{12}$
\Comment{Hence, without loss of generality, $\hat A_i^{12}= \chi_{i}^k$}
\State \label{en:+1}
Decompose $\v^k=\sum_j \lambda_{jk}e_j^k\otimes\v^{12}$
\EndFor
\State \label{en:+2}
Find an orthogonal matrix $U$ sending $e_1^1\otimes\sum_{j}\lambda_{j1}e_j^1 \mapsto e_1^2\otimes\sum_{j}\lambda_{j2}e_j^2 $
\For{$i\in\{1,\dots,n\}$}
\State 
$A_i :=
\begin{cases}
\I_{r_2} \otimes \hat A_{i}^1   & \text{if } i\in I_1,\\
 (U^T\otimes\I_{r_{12}})(\I_{r_1} \otimes  \hat A_{i}^2)(U\otimes\I_{r_{12}})& \text{if } i\in I_2
\end{cases}
$ \label{en:5} 
\EndFor
\State\label{en:+3} Compute $\v = e_1^1\otimes\v^1$
\Ensure $\underline{A} = (A_1,\dots,A_n)$ and $\v$.
\end{algorithmic}
\end{algorithm}
\begin{corollary}
\label{coro:sparse_gns_sound}
The procedure $\sparsegns$ described in Algorithm~\ref{algorithm:sparse_gns} is sound and returns the tuple $\underline{A}$ and the vector $\v$ from Theorem~\ref{th:sparse_flat}.
\end{corollary}
\begin{proof}
Correctness of the algorithm has been essentially established in the proof of Theorem \ref{th:sparse_flat}.
Both computation in Line~\ref{en:3} and Line~\ref{en:4} can be performed since the only homomorphisms out of full matrix algebras are ampliations composed with a conjugation (by the Skolem-Noether theorem).
One can perform an orthogonal change of basis in Line~\ref{en:3b}, and $\hat A_i^k= \I \otimes\chi_{i}^k$, for all $k\in \{1,2\}$ and $i \in I_1 \cap I_2$. 
Indeed, let us assume that a matrix $P$ is invertible, and the map $\phi:A\mapsto P^{-1}AP$  from $\Mbb_n(\R)$ to $\Mbb_n(\R)$ preserves transposes. 
Then, the following equalities 
\[
\begin{split}
\phi(A^T) = P^{-1} A^T P = (P^{-1}AP)^T= P^T A^T P^{-T}
\end{split}
\]
imply that $PP^T$ commutes with all $n\times n$ matrices. Therefore, $PP^T$ is a scalar matrix, and $P$ is a scalar multiple of an orthogonal matrix, the desired result. 
%
Eventually, each component of the tuple $A$, given in Line~\ref{en:5}, is well defined by construction and gives 
rise to the desired amalgamation. Line \ref{en:+3} constructs the vector $\v$ needed for \eqref{eq:desired} to hold.
\end{proof}
\section{Eigenvalue Optimization of Noncommutative Sparse  Polynomials}
\label{sec:eig}

The aim of this section is to provide SDP relaxations allowing one to under-approximate the smallest eigenvalue that a given nc polynomial can attain on a tuple of symmetric matrices from a given semialgebraic set.
The unconstrained case is handled in Section~\ref{sec:unconstr_eig}, where we show how to compute a lower bound on the smallest eigenvalue via solving an SDP.
The constrained case is handled in Section~\ref{sec:constr_eig}, where we derive a hierarchy of lower bounds converging to the minimal eigenvalue, assuming that the quadratic module is archimedean and that RIP holds (Assumption~\ref{hyp:sparsityRIP}).

We first recall the celebrated Helton-McCullough Sums of Squares theorem~\cite{Helton02,McCullSOS} stating the equivalence between sums of hermitian squares (SOHS) and positive semidefinite nc polynomials. 
\begin{theorem}
\label{th:Helton}
Given $f \in \RX$, we have $f(\underline{A}) \succeq 0$, for all $\underline{A} \in \Sbb^n$, if and only if $f \in \SigmaX$.
\end{theorem}
In contrast with the constrained case where we obtain the analog of Putinar's Positivstellensatz in Theorem~\ref{th:sparsePsatz}, there is no sparse analog of Theorem \ref{th:Helton}, as shown in the following example. 
{
\begin{lemma}
\label{lemma:nosparseHelton}
There exist polynomials which are sparse sums of hermitian squares but are not sums of sparse hermitian squares.
\end{lemma}
\begin{proof}
Let $v=\begin{bmatrix}X_1&X_1X_2&X_2&X_3&X_3X_2 \end{bmatrix}$, 
\begin{equation}\label{eq:gram}
G=\left[\begin{array}{rrrrr}
1&-1&-1&0&\alpha\\
-1&2&0&-\alpha&0\\
-1&0&3&-1&9\\
0&-\alpha&-1&6&-27\\
\alpha&0&9&-27&142
 \end{array}\right],\qquad \alpha\in\R,
\end{equation} 
and consider
\begin{equation}\label{eq:polySohs}
\begin{split}
f&=vGv^\star \\
&= X_1^2-X_1 X_2-X_2 X_1+3
   X_2^2-2
   X_1 X_2 X_1+2 X_1 X_2^2 X_1 \\ 
   &\phantom{=\ }
-X_2 X_3-X_3 X_2+6 X_3^2
   +9 X_2^2 X_3+9 X_3  X_2^2-54
   X_3 X_2 X_3 +142
   X_3 X_2^2 X_3.
\end{split}
\end{equation}
The polynomial $f$ is clearly sparse w.r.t.~$I_1=\{x_1,x_2\}$ and $I_2=\{x_2,x_3\}$. Note that the matrix $G$ is positive semidefinite if and only if  $0.270615 \lesssim \alpha \lesssim 1.1075$, whence $f$  is a sparse polynomial that is an SOHS.

We claim that $f\not\in\SigmaXI1+\SigmaXI2$, i.e., $f$ is not a sum of sparse hermitian squares. By the Newton chip method \cite[Section 2.3]{burgdorf16} only monomials in $v$ can appear in a sum of squares decomposition of $f$. Further, every Gram matrix of $f$ (with border vector $v$) is of the form \eqref{eq:gram}. However, the matrix $G$ with $\alpha=0$ is not positive semidefinite, hence $f\not\in\SigmaXI1+\SigmaXI2$.
\end{proof}
}
\subsection{Unconstrained Eigenvalue Optimization with Sparsity}
\label{sec:unconstr_eig}
~\\
Let $\I$ stands for the identity matrix. 
Given $f \in \SymRX$ of degree $2 d$, the smallest eigenvalue of $f$ is obtained by solving the following  optimization problem
\begin{align}
\label{eq:eigmin}
\lambda_{\min}(f) := \inf \{\langle f(\underline{A}) \v,  \v \rangle : \underline{A} \in \Sbb^n, \| \v \| = 1 \} \,.
\end{align}
The optimal value $\lambda_{\min}(f)$ of Problem~\eqref{eq:eigmin} is the greatest lower bound on the eigenvalues of $f(\underline{A})$ over all $n$-tuples $\underline{A}$ of real symmetric matrices. 
Problem~\eqref{eq:eigmin} can be rewritten as follows:
\begin{equation}
\label{eq:eigmin2}
\begin{aligned}
\lambda_{\min}(f) = \sup\limits_{\lambda} \quad  & \lambda  \\	
\text{s.t.} 
\quad & f(\underline{A}) - \lambda \I \succeq 0  \,, \quad \forall \underline{A} \in \Sbb^n \,,
\end{aligned}
\end{equation}
which is in turn equivalent to
\begin{equation}
\label{eq:eigmin_dual}
\begin{aligned}
\lambda_{\min,d}(f) = \sup\limits_{\lambda} \quad  & \lambda  \\	
\text{s.t.} 
\quad & f(\underline{X}) - \lambda \in \SigmaX_{d} \,,
\end{aligned}
\end{equation}
as a consequence of Theorem~\ref{th:Helton}.

The dual of SDP~\eqref{eq:eigmin_dual} is 
\begin{equation}
\label{eq:eigmin_primal}
\begin{aligned}
L_{\sohs,d}(f) = \inf\limits_{\newnc{L}} \quad  & \langle {\newnc{\M_d}(L)}, G_f \rangle  \\	
\text{s.t.} 
\quad & \newnc{L(1) = 1} \,, \quad {\newnc{\M_d}(L)} \succeq 0 \,,\\
\quad & \newnc{L : \RX_{2 d} \to \R \, 
 \text{ linear}}\,,
\end{aligned}
\end{equation}
%
where $G_f$ is a Gram matrix for $f$ (see Proposition~\ref{prop:ncGram}).\\ 
One can compute $\lambda_{\min}(f)$ by solving a single SDP,  either SDP~\eqref{eq:eigmin_primal} or 
SDP~\eqref{eq:eigmin_dual}, since there is no duality gap between these two programs (see, e.g.,~\cite[Theorem~4.1]{burgdorf16}), that is, one has $L_{\sohs,d}(f) = \lambda_{\min,d}(f) = \lambda_{\min}(f)$.

Now, we address eigenvalue optimization for a given  sparse nc polynomial $f = f_1 + \dots + f_p$ of degree $2 d$, with $f_k \in \SymRXI{k}_{2d}$, for all $k=1,\dots,p$.
For all $k=1,\dots,p$, let $G_{f_k}$ be a Gram matrix associated to $f_k$.
The sparse variant of SDP~\eqref{eq:eigmin_primal} is
\begin{equation}
\label{eq:sparse_eigmin_primal}
\begin{aligned}
L^{\sparse}_{\sohs,d}(f) = \inf\limits_{\newnc{L}} \quad  & \sum_{k=1}^p \langle {\newnc{\M_d}(L,I_k)}, G_{f_k} \rangle  \\	
\text{s.t.} 
\quad & \newnc{L(1) = 1} \,, \quad {\newnc{\M_d}(L,I_k)}  \succeq 0 \,, \quad k=1,\dots,p \,, \\ 
\quad & \newnc{L : \RXI{1}_{2 d} + \dots + \RXI{p}_{2 d} \to \R \,
\text{ linear}}\,,
\end{aligned}
\end{equation}
%
whose dual is the sparse variant of SDP~\eqref{eq:eigmin_dual}:
\begin{equation}
\label{eq:sparse_eigmin_dual}
\begin{aligned}
\lambda^{\sparse}_{\min,d}(f) = \sup\limits_{\lambda} \quad  & \lambda  \\	
\text{s.t.} 
\quad & f - \lambda  \in \SigmaXI{1}_{2 d} + \dots + \SigmaXI{p}_{2 d} \,,
\end{aligned}
\end{equation}
To prove that there is no duality gap between SDP~\eqref{eq:sparse_eigmin_primal} and SDP~\eqref{eq:sparse_eigmin_dual}, we need a sparse variant of~\cite[Proposition~3.4]{mccullough2005noncommutative}, which says that $\SigmaX_d$ is closed in $\RX_{2 d}$:
\begin{proposition}
\label{prop:sparseclosed}
The set $\SigmaX^{\sparse}_d$ is a closed convex subset of $\RXI{1}_{2 d} + \dots + \RXI{p}_{2 d}$.
\end{proposition}
\begin{proof}
For each $k \in \{1,\dots,p\}$, we endow each $\RXI{k}_{2 d}$ with a norm $\|\cdot\|_k$. 
For each $f\in \RXI{1}_{2d} + \dots + \RXI{p}_{2d}$,  we set 
\[
\| f \| := \inf \left\lbrace \| f_1 \|_1 + \dots + \| f_p \|_2 : f = f_1 + \dots + f_p, f_k \in \RXI{k}_{2d} \right\rbrace
\]
Let us consider an element $h = h_1 + \dots + h_p \in \SigmaX^{\sparse}_d$, with $h_k \in \SigmaXI{k}_d$.
For each $k \in \{1,\dots,p\}$, $h_k$ can be written as a sum of at most $\bsigma(n_k,d)$ hermitian squares of degree at most $2 d$ by Proposition~\ref{prop:ncGram}. 
Define the mapping
\begin{align*}
\phi_k : (\RXI{k}_d)^{\bsigma(n_k,d)} & \to \RXI{k}_{2d}\\
(h_{k j})_{j=1}^{\bsigma(n_k,d)} & \mapsto \sum_{j=1}^{\bsigma(n_k,d)} h_{kj}^\star h_{k j} \,,
\end{align*}
and let us denote $\h_k = (h_{k j})_{j=1}^{\bsigma(n_k,d)}$.
Then, the image of the map $\phi$, defined by
\begin{align*}
\phi : \prod_{k=1}^p (\RXI{k}_d)^{\bsigma(n_k,d)}  & \to \RXI{1}_{2d} + \dots + \RXI{p}_{2d} \\
(\h_1, \dots,\h_p) & \mapsto \phi_1(\h_1) + \dots + \phi_p(\h_p) \,,
\end{align*}
is equal to $\SigmaX^{\sparse}_d$. 

Let us define the subset $\mathcal{V} \subset \prod_{k=1}^p (\RXI{k}_d)^{\bsigma(n_k,d)} $ by 
\[
\mathcal{V} := \left\lbrace \h = \h_1 + \dots + \h_p : 
\sum_{k=1}^p
\left\| \sum_{j=1}^{\bsigma(n_k,d)} h_{k j}^\star h_{k j} \right\|_k
= 1 \right\rbrace \,.
\] 
Since  $\mathcal{V}$ is compact, then $\phi (\mathcal{V})$ is also compact. 
Note that $0 \notin \mathcal{V}$, implying that $0 \notin \phi (\mathcal{V})$.
Next, let us consider a sequence $(f^{\ell})_{\ell \geq 1}$ of elements of $\SigmaX^{\sparse}_d$, converging to $f \in \RXI{1}_{2d} + \dots + \RXI{p}_{2d}$.
One can write $f^{\ell} = \lambda^{\ell} v^{\ell}$ for $\lambda^{\ell} \in \R^{\geq 0}$ and $v^{\ell} \in \phi(\mathcal{V})$. 
By compactness of $\phi (\mathcal{V})$, there exists a subsequence of $(v^{\ell})_{\ell}$, (also denoted $(v^{\ell})_{\ell}$), which converges to $v \in \phi (\mathcal{V})\subset \SigmaX^{\sparse}_d$. 
By definition of $\|\cdot\|$ and $\mathcal{V}$, one has $\|v^{\ell}\| \leq 1$, for all $\ell \geq 1$. 
Since $0 \notin \phi(\mathcal{V})$ and $\phi(\mathcal{V})$ is compact, there exists an $\varepsilon > 0$ such that $\|v^{\ell}\| > \varepsilon$, for all $\ell \geq 1$.
Therefore,  $\lambda^{\ell} = \frac{\|f^{\ell}\|}{\|v^{\ell}\|}$ converges to $ \frac{\|f\|}{\|v\|}$, as $\ell$ goes to infinity.
From this, we deduce that $f^{\ell}$ converges to $f = \frac{\|f\|}{\|v\|} v \in \SigmaX^{\sparse}_d$, yielding the desired result.
\end{proof}
From Proposition~\ref{prop:sparseclosed}, we obtain the following theorem which does not require Assumption~\ref{hyp:sparsityRIP}.
\begin{theorem}
\label{th:sparse_eig_nogap}
Let $f \in \SymRX$ of degree $2 d$, with $f = f_1 + \dots + f_p$, $f_k \in \SymRXI{k}_{2d}$, for all $k=1,\dots,p$. 
Then, one has $\lambda^{\sparse}_{\min,d}(f) = L^{\sparse}_{\sohs,d}(f)$, i.e., there is no duality gap between SDP~\eqref{eq:sparse_eigmin_primal} and SDP~\eqref{eq:sparse_eigmin_dual}.
\end{theorem}
\begin{proof}
The strong duality is obtained exactly as for the dense case~\cite[Theorem~4.1]{burgdorf16}, and relies on the closedness of $\SigmaX^{\sparse}_d$, stated in Proposition~\ref{prop:sparseclosed}.
\if{
By weak duality, one has $\lambda^{\sparse}_{\min,d}(f)  \leq L^{\sparse}_{\sohs,d}(f)$. 
Then, one has $L_{\sohs,d}(f) = \lambda_{\min,d}(f)$ by~\cite[Theorem~4.1]{burgdorf16}.
Further, note that SDP~\eqref{eq:sparse_eigmin_primal} is a relaxation of SDP~\eqref{eq:eigmin_primal}, implying that $L^{\sparse}_{\sohs,d}(f) \leq L_{\sohs,d}(f)$.
By combining the above (in)equalities, we obtain 
\[
L^{\sparse}_{\sohs,d}(f) \leq L_{\sohs,d}(f) = \lambda_{\min,d}(f) = \lambda^{\sparse}_{\min,d}(f) \leq L^{\sparse}_{\sohs,d}(f) \,,
\]
yielding the desired result.
}\fi
\end{proof}
\begin{remark} \rm
\label{rk:sparsevsdense}
By contrast with the dense case, it is not enough to compute the solution of SDP~\eqref{eq:sparse_eigmin_primal} to obtain the optimal value $\lambda_{\min}(f)$ of the unconstrained optimization problem~\eqref{eq:eigmin}.
However, one can still compute a certified lower bound  $\lambda^{\sparse,d}_{\min}(f)$ by solving a single SDP, either in the primal form~\eqref{eq:sparse_eigmin_primal} or in the dual form~\eqref{eq:sparse_eigmin_dual}. 
Note that the related computational cost is potentially much less expensive. 
Indeed, SDP~\eqref{eq:sparse_eigmin_dual} involves $\sum_{k=1}^p \bsigma(n_k,2 d)$ equality constraints and  $\sum_{k=1}^p\bsigma(n_k,d)+1$ variables. 
This is in contrast with the dense version~\eqref{eq:eigmin_dual}, which involves $\bsigma(n,2 d)$ equality constraints and  $1 + \bsigma(n,d)$ variables. 
\end{remark}
%
\subsection{Constrained Eigenvalue Optimization with Sparsity}
\label{sec:constr_eig}
~\\
Here, we focus on providing lower bounds for the constrained eigenvalue optimization of nc polynomials. 
Given $f \in \SymRX$ and $S := \{g_1,\dots,g_{m} \} \subset \SymRX$ as in~\eqref{eq:DS}, let us define  $\lambda_{\min} (f, S)$ as follows:
\begin{align}
\label{eq:constr_eigmin}
\lambda_{\min}(f,S) := \inf \{\langle f(\underline{A}) \v , \v \rangle : \underline{A} \in {D_S^\infty}, \| \v \| = 1 \} \,,
\end{align}
which is, as for the unconstrained case, equivalent to 
\begin{equation}
\label{eq:constr_eigmin2}
\begin{aligned}
\lambda_{\min}(f,S) = \sup\limits_{\lambda} \quad  & \lambda  \\	
\text{s.t.} 
\quad & f(\underline{A}) - \lambda \I \succeq 0  \,, \quad \forall \underline{A} \in {D_S^\infty} \,.
\end{aligned}
\end{equation}
Let $d_j := \lceil \deg g_j / 2 \rceil$, for each $j=1,\dots,m$ and $d := \max \{\lceil \deg f / 2 \rceil, d_1, \dots, d_m \}$.
As shown in~\cite{pironio2010convergent,cafuta2012constrained}
(see also \cite{Helton04}), one can approximate $\lambda_{\min} (f, S)$ from below via the following hierarchy of SDP programs, indexed by $s \geq d$:
\begin{equation}
\label{eq:constr_eigmin_dual}
\begin{aligned}
\lambda_s (f,S) := \sup\limits_{\lambda} \quad  & \lambda  \\	
\text{s.t.} 
\quad & f - \lambda \in {\cM(S)_s} \,.
\end{aligned}
\end{equation}
The dual of SDP~\eqref{eq:constr_eigmin_dual} is 
\begin{equation}
\label{eq:constr_eigmin_primal}
\begin{aligned}
L_s(f,S) := \inf\limits_{\newnc{L}} \quad  & \langle {\newnc{\M_s}(L)}, G_f \rangle  \\	
\text{s.t.} 
\quad & \newnc{L(1) = 1} \,, \\
\quad & {\newnc{\M_s}(L)} \succeq 0 \,, 
\quad {\newnc{\M_{s-d_j}}(g_j L)} \succeq 0 \,, \quad j=1,\dots,m \,, \\ 
\quad & \newnc{L : \RX_{2 d}  \to \R \,
\text{ linear}}\,,
\end{aligned}
\end{equation}
%
%
Under additional assumptions, this hierarchy of primal-dual SDP~\eqref{eq:constr_eigmin_dual}-\eqref{eq:constr_eigmin_primal} converges to the value of the constrained eigenvalue problem.
\begin{corollary}
\label{th:constr_eig}
\newnc{Assume that ${D_S}$ is as in~\eqref{eq:newDS} with the additional quadratic constraints~\eqref{eq:additional}} and that the quadratic module $M_S$ is archimedean. 
Then the following holds for each $f \in \SymRX$:
\begin{align}
\label{eq:cvg_eig}
\lim_{s \to \infty} L_s(f,S) = \lim_{s \to \infty}  \lambda_s (f,S) = \lambda_{\min} (f, S) \,.
\end{align}
\end{corollary}
The main ingredient of the proof (see, e.g.,~\cite[Corollary~4.11]{burgdorf16}) is the nc analog of Putinar's Positivstellensatz, stated in Theorem~\ref{th:densePsatz}.\\
Let $S \cup \{f\} \subseteq \SymRX$  and let ${D_S}$ be as in~\eqref{eq:newDS} with the additional quadratic constraints~\eqref{eq:additional}.  
Let ${\cM(S)^{\sparse}}$ be as in~\eqref{eq:sparseMS} and let us define ${\cM(S)_s^{\sparse}}$ in the same way as the truncated quadratic module ${\cM(S)_s}$ in~\eqref{eq:MS2d}.
Now, let us state the sparse variant of the primal-dual hierarchy~\eqref{eq:constr_eigmin_dual}-\eqref{eq:constr_eigmin_primal} of lower bounds for $\lambda_{\min} (f, S)$. 

For all $s \geq d$, the sparse variant of SDP~\eqref{eq:constr_eigmin_primal} is
\begin{equation}
\label{eq:sparse_constr_eigmin_primal}
\begin{aligned}
L_s^{\sparse}(f,S) := \inf\limits_{\newnc{L}} \quad  & \sum_{k=1}^p \langle {\newnc{\M_s}(L,I_k)}, G_{f_k} \rangle  \\	
\text{s.t.} 
\quad & \newnc{L(1) = 1} \,, \\
\quad & {\newnc{\M_s}(L,I_k)} \succeq 0 \,, \quad k=1,\dots,p \,, \\ 
\quad &  {\newnc{\M_{s-d_j}}(g_jL,I_k)} \succeq 0 \,,  \quad j=1,\dots,m \,, \quad k=1,\dots,p \,, \\ 
\quad & \newnc{L : \RXI{1}_{2 d} + \dots + \RXI{p}_{2 d} \to \R \,
\text{ linear}}\,,
\end{aligned}
\end{equation}
whose dual is the sparse variant of SDP~\eqref{eq:constr_eigmin_dual}:
\begin{equation}
\label{eq:sparse_constr_eigmin_dual}
\begin{aligned}
\lambda_s^{\sparse} (f,S) := \sup\limits_{\lambda} \quad  & \lambda  \\	
\text{s.t.} 
\quad & f - \lambda \in {\cM(S)_s} \,.
\end{aligned}
\end{equation}
Recall that an $\varepsilon$-neighborhood of 0 is the set $\mathcal{N}_{\varepsilon}$ defined for a given $\varepsilon > 0$ by:
\[
\mathcal{N}_{\varepsilon} := \bigcup_{k \in \N} \left\lbrace \underline{A} := (A_1,\dots,A_n) \in \Sbb_k^n : \varepsilon^2 - \sum_{i=1}^n A_i^2 \succeq 0  \right\rbrace \,.
\]
\begin{lemma}
\label{lemma:epsneighborhood}
If $h \in \RX$ vanishes on an $\varepsilon$-neighborhood of 0, then $h = 0$.
\end{lemma}
\begin{proof}
See~\cite[Lemma~1.35]{burgdorf16}.
\end{proof}
\begin{proposition}
\label{prop:constr_eigmin_slater}
Let $S \cup \{f\} \subseteq \SymRX$,  assume that ${D_S}$ contains an $\varepsilon$-neighborhood of 0 and that ${D_S}$ is as in~\eqref{eq:newDS} with the additional quadratic constraints~\eqref{eq:additional}. 
Then SDP~\eqref{eq:sparse_constr_eigmin_primal} admits strictly feasible solutions. 
\end{proposition}
\begin{proof}
This proof being almost the same as the one of~\cite[Proposition~4.9]{burgdorf16} is presented for the sake of completeness.
By Lemma~\ref{lemma:Hankel}, it is enough to build a linear map $L : \SymRX_{2 s} \to \R $ such that for all $k=1,\dots,p$ one has: 
\begin{itemize}
\item $L(h^\star h) > 0$, for all nonzero $h \in \RXI{k}_s$;
\item for all $j \in J_k$, one has $L(h^\star g_j h) > 0$, for all nonzero $h \in  \RXI{k}_{s - \lceil \deg g_j / 2 \rceil }$.
\end{itemize}

Let us pick $N > s$ and let $\mathcal{U}$ stands for the set of all $N \times N$ matrices from ${D_S}$ with rational entries:
\[
\mathcal{U} := \{ { \underline{A}^{(r)} := (A_1^{(r)},\dots, A_n^{(r)}} ) : r \in \N \,,  \underline{A}^{(r)} \in {D_S^N} \}
\]
Note that this set $\mathcal{U}$ contains a dense subset of 
$\mathcal{N}_{\varepsilon}$. 
Let us associate to $\underline{A} \in \mathcal{U}$ the linear map $L_{\underline{A}} : \SymRX_{2 d} \to \R$ defined by $L_{\underline{A}}(h) := \Trace (h(\underline{A}))$.
From this, we define $L$ as follows:
\[ 
L := \sum_{r=1}^\infty 2^{-r} \frac{L_{\underline{A}^{(r)}}}{\| L_{\underline{A}^{(r)}} \|} \,.
\]
Now let us fix $k \in \{1,\dots,p \}$.
Obviously, one has $L(h^\star h) \geq 0$, for all nonzero $h \in \RXI{k}_d$. 
Let us suppose that $L(h^\star h) = 0$ for some $h \in \RXI{k}_d$.
Then, one has $L_{\underline{A}^{(r)}} (h^\star h) =  0 = \Trace (  h^\star(\underline{A}^{(r)})  h(\underline{A}^{(r)}) )$, for all $r \in \N$.
This implies that for all $r \in \N$, one has $h^\star(\underline{A}^{(r)})  h(\underline{A}^{(r)}) = 0$, which in turn yields $h(\underline{A}^{(r)}) = 0$.
Since $\mathcal{U}$ contains a dense subset of $\mathcal{N}_{\varepsilon}$,  this implies that $h$ vanishes on a $\varepsilon$-neighborhood of 0. 
As a consequence of Lemma~\ref{lemma:epsneighborhood}, one has $h = 0$. 

In a similar way, we prove that if $L(h^\star g_j h) = 0$ for some $h \in  \RXI{k}_{s - \lceil \deg g_j / 2 \rceil }$, then one necessarily has $h = 0$.
\end{proof}
\begin{corollary}
\label{th:sparse_constr_eig}
Let $S \cup \{f\} \subseteq \SymRX$, assume that  ${D_S}$ is as in~\eqref{eq:newDS} with the additional quadratic constraints~\eqref{eq:additional}. 
Let Assumption~\ref{hyp:sparsityRIP} hold. 
Then, one has
\begin{align}
\label{eq:sparse_cvg_eig}
\lim_{s \to \infty} L_s^{\sparse}(f,S) = \lim_{s \to \infty}  \lambda_s^{\sparse} (f,S) = \lambda_{\min} (f, S) \,.
\end{align}
\end{corollary}
\begin{proof}
The proof is similar to the one in the dense case.
Let us take $\lambda := \lambda_{\min} (f, S) - \varepsilon$.
Then, one has $f - \lambda \succ 0$ on ${D_S^{\infty}}$, so Theorem~\ref{th:sparsePsatz} implies that $f - \lambda \in {\cM(S)^{\sparse}}$. 
Hence, there exists \newnc{$s$} such that $f - \lambda \in {\cM(S)_s^{\sparse}}$, yielding a feasible solution for SDP~\eqref{eq:sparse_constr_eigmin_dual}, \newnc{so $\lambda_{\min} (f, S) - \varepsilon \leq \lambda_s^{\sparse} (f,S)$. 
By weak duality between SDP~\eqref{eq:sparse_constr_eigmin_primal} and SDP~\eqref{eq:sparse_constr_eigmin_dual}, one has  $\lambda_s^{\sparse} (f,S) \leq L_s^{\sparse}(f,S) $.
Therefore, one obtains $\lambda_{\min} (f, S) - \varepsilon \leq \lambda_s^{\sparse} (f,S) \leq L_s^{\sparse}(f,S) \leq \lambda_{\min} (f, S)$, yielding the desired result. }
\end{proof}
As for the unconstrained case, there is no sparse variant of the ``perfect'' Positivstellensatz stated in~\cite[\S4.4]{burgdorf16} or \cite{hkmConvex}, for constrained eigenvalue optimization over convex nc semialgebraic sets, such as those associated either to the sparse nc ball $\Bbb^{\sparse} := \{1 - \sum_{i \in I_1} X_i^2, \dots, 1 - \sum_{i \in I_p} X_i^2 \}$ or the nc polydisc $\Dbb := \{1 - X_1^2,\dots,1-X_n^2 \}$.
Namely, for an nc polynomial $f$ of degree $2 d + 1$,
 computing only SDP~\eqref{eq:sparse_eigmin_primal} with optimal value $\lambda^{\sparse}_{\min,d+1}(f, S)$ when $S = \Bbb^{\sparse}$ or $S = \Dbb^{\sparse}$ 
 does not suffice
 to obtain the value of $\lambda_{\min}(f,S)$.
This is explained in Example~\ref{ex:nosparseEigBall} below, which implies that there is no sparse variant of~\cite[Corollary~4.18]{burgdorf16} when $S = \Bbb^{\sparse}$.
\begin{example}\rm
\label{ex:nosparseEigBall}
Let us consider a randomly generated cubic polynomial $f = f_1 + f_2$ with
\begin{align*}
f_1 = \ & 4 - X_1 + 3X_2 - 3X_3 - 3X_1^2 - 7X_1X_2 +   6X_1X_3 - X_2X_1 -5X_3X_1 + 5X_3X_2 \\
& -  5X_1^3 - 3X_1^2 X_3 + 4X_1X_2X_1 -  6X_1X_2X_3 + 7X_1X_3X_1 
 + 2X_1X_3X_2 -   X_1X_3^2 \\
 & - X_2X_1^2 + 3X_2X_1X_2 -   X_2X_1X_3 - 2 X_2^3 - 5 X_2^2 X_3 
 -   4X_2X_3^2 - 5X_3X_1^2 \\
 & + 7X_3X_1X_2 +   6X_3X_2X_1 - 4X_3X_2X_2 - X_3^2 X_1 -   2X_3^2 X_2 + 7X_3^3 \,, \\
f_2  = \ & -1 + 6X_2 + 5X_3 + 3X_4 - 5X_2^2 + 2X_2X_3 +   4X_2X_4 - 4X_3X_2 + X_3^2 - X_3X_4 \\
& +   X_4X_2 - X_4X_3 + 2X_4^2 - 7X_2^3 +   4X_2X_3^2 + 5X_2X_3X_4 - 7X_2X_4X_3 -   7X_2X_4^2 \\
& + X_3X_2^2 + 6X_3X_2X_3 -   6X_3X_2X_4 - 3X_3^2 X_2 - 7X_3^2X_4 +   6X_3X_4X_2 \\
& - 3X_3X_4X_3 - 7X_3X_4^2 +   3X_4X_2^2 - 7X_4X_2X_3 - X_4X_2X_4 -   5X_4X_3^2  \\
& + 7X_4X_3X_4 + 6X_4^2 X_2 -   4 X_4^3 \,,
\end{align*} 
and the nc polyball $S = \Bbb^{\sparse} = \{1-X_1^2-X_2^2-X_3^2, 1-X_2^2-X_3^2 - X_4^2 \}$ corresponding to $I_1 = \{1,2,3\}$ and $I_2 = \{2,3,4\}$.
Then, one has $
\lambda_2^{\sparse} (f, S ) \simeq 
-27.536 < \lambda_3^{\sparse} (f,S) \simeq -27.467
\simeq \lambda_{\min,2} (f,S) = \lambda_{\min} (f,S)$.
\end{example}

\subsection{Extracting Optimizers}
\label{sec:gns_eig}
Here, we explain how to extract a pair 
of optimizers 
$(\underline{A},\v)$ for the eigenvalue optimization problems 
when the flatness and irreducibility conditions of Theorem~\ref{th:sparse_flat} hold. 
We apply the $\sparsegns$ procedure described in Algorithm~\ref{algorithm:sparse_gns} on the optimal solution of SDP~\eqref{eq:sparse_eigmin_primal} in the unconstrained case or SDP~\eqref{eq:sparse_constr_eigmin_primal} in the constrained case.
In the unconstrained case, we have the following sparse variant of~\cite[Proposition~4.4]{burgdorf16}.
\begin{proposition}
\label{prop:sparse_eig_flat}
Given $f$ as in Theorem~\ref{th:sparse_eig_nogap}, let us assume that SDP~\eqref{eq:sparse_eigmin_primal} yields an optimal  solution ${\newnc{\M_{d+1}}(L)}$ associated to $L^{\sparse}_{\sohs,d+1}(f)$.
If the linear functional $L$ underlying ${\newnc{\M_{d+1}}(L)}$ satisfies the  flatness (H1) and irreducibility (H2) conditions stated in Theorem~\ref{th:sparse_flat}, then one has 
\[ \lambda_{\min} (f) = L^{\sparse}_{\sohs,d+1}(f) = \sum_{k=1}^p \langle {\newnc{\M_{d+1}}(L,I_k)}, G_{f_k} \rangle 
\,. \]
\end{proposition}
\begin{proof}
The first equality comes from Theorem~\ref{th:sparse_eig_nogap}.
Let us assume that each moment matrix satisfies the assumptions of Theorem~\ref{th:sparse_flat}. 
Then, we obtain a tuple $\underline{A}$ of symmetric matrices and a unit vector $\v$ such that $L(f) = \langle f(\underline{A}) \v \mid \v \rangle$. 
Since one has $L(f) = \sum_{k=1}^p \langle \newnc{\M_{d+1}}(L,I_k), G_{f_k} \rangle = \lambda_{\min} (f)$, the desired result holds.
\end{proof}
We can extract optimizers for the unconstrained minimal eigenvalue problem~\eqref{eq:eigmin} thanks to the following algorithm.

\begin{algorithm}
$\sparseeiggns$
\label{algorithm:sparse_eig}
\begin{algorithmic}[1]
\Require $f \in \SymRX_{2d}$ satisfying Assumption~\ref{hyp:sparsityRIP}.
\State Compute $L^{\sparse}_{\sohs,d+1}(f)$ by solving SDP~\eqref{eq:sparse_eigmin_primal} 
\If{SDP~\eqref{eq:sparse_eigmin_primal}  is unbounded or its optimum is not attained}
\State Stop
\EndIf 
\State Let $\newnc{\M_{d+1}}(L)$ be an optimizer of SDP~\eqref{eq:sparse_eigmin_primal}.  Compute $\underline{A},\v := \sparsegns \,(\newnc{\M_{d+1}}(L))$.
\Ensure $\underline{A}$ and $\v$.
\end{algorithmic}
\end{algorithm}
In the constrained case, the next result is the sparse variant of~\cite[Theorem~4.12]{burgdorf16} and is a direct corollary of Theorem~\ref{th:sparse_flat}.
\begin{corollary}
\label{th:sparse_cons_eig}
Let $S \cup \{f\} \subseteq \SymRX$, assume that  ${D_S}$ is as in~\eqref{eq:newDS} with the additional quadratic constraints~\eqref{eq:additional}. 
Suppose Assumptions~\ref{hyp:sparsityRIP}(i)-(ii) hold. 
Let ${\newnc{\M_s}(L)}$ be an optimal solution of SDP~\eqref{eq:sparse_constr_eigmin_primal} with value $L_s(f,S)$, for $s \geq d + \delta$, such that $L$ satisfies the assumptions of Theorem~\ref{th:sparse_flat}.
Then, there exist $r \in \N$, $\underline{A} \in {D_S^r}$ and a unit vector $\v$ such that
\[
\lambda_{\min}(f,S) = 
\langle f(\underline{A}) \v , \v \rangle = 
L_s(f,S) \,.
\] 
\end{corollary}
\begin{remark} \rm
As in the dense case~\cite[Algorithm~4.2]{burgdorf16}, one can provide a randomized algorithm to look for flat optimal solutions for the constrained eigenvalue problem~\eqref{eq:constr_eigmin}.
The underlying reason which motivates this randomized approach is  work by Nie, who derives in~\cite{NieRand14} a hierarchy of SDP programs, with a random objective function, that converges to a flat solution (under mild assumptions).
\end{remark}

\begin{example}\rm
Consider the sparse polynomial $f=f_1+f_2$ from Example~\ref{ex:nosparseEigBall}. The Hankel matrix ${\newnc{\M_3}(L)}$ obtained when computing
$\lambda_3^{\sparse}$ 
by solving \eqref{eq:sparse_constr_eigmin_primal} for $s=3$
satisfies the  flatness (H1) and irreducibility (H2) conditions of Theorem~\ref{th:sparse_flat}.
 We can thus apply the $\sparsegns$ algorithm yielding
 \begin{align*}
A_1&=\left[\begin{array}{rrrr}
   0.0059 &  0.0481 &   0.1638&    0.4570\\
    0.0481&  -0.2583 &   0.5629&   -0.2624\\
    0.1638&   0.5629  &  0.3265 &  -0.3734\\
    0.4570&  -0.2624   &-0.3734  & -0.2337
\end{array}\right] \\
A_2&=\left[\begin{array}{rrrr}
   -0.3502 &   0.0080&    0.1411&    0.0865\\
    0.0080  & -0.4053 &   0.2404 &  -0.1649\\
    0.1411   & 0.2404  & -0.0959  &  0.3652\\
    0.0865   &-0.1649   & 0.3652   & 0.4117
\end{array}\right] \\
A_3&=\left[\begin{array}{rrrr}
 -0.7669 &  -0.0074 &  -0.1313  & -0.0805\\
   -0.0074&   -0.4715&   -0.2238 &   0.1535\\
   -0.1313 &  -0.2238 &   0.0848  & -0.3400\\
   -0.0805  &  0.1535  & -0.3400   &-0.2126
\end{array}\right] \\
A_4&=\left[\begin{array}{rrrr}
    0.3302&   -0.1839&    0.1811&   -0.0404\\
   -0.1839 &  -0.1069 &   0.5114 &  -0.0570\\
    0.1811  &  0.5114  &  0.1311  & -0.3664\\
   -0.0404   &-0.0570   &-0.3664   & 0.4440
\end{array}\right] 
\end{align*}
 where
 \[
f(\underline A)
=\left[\begin{array}{rrrr}
   -10.3144 &   3.9233 &  -5.0836 &   -7.7828\\
    3.9233  &  1.8363  &  4.5078 &  -7.5905\\
   -5.0836   & 4.5078   &-19.5827  &  13.9157\\
   -7.7828   &-7.5905    &13.9157   & 8.3381
\end{array}\right] 
 \]
 has  minimal eigenvalue $-27.4665$ with unit eigenvector
 \[\v=\begin{bmatrix} 0.1546 &  -0.2507 &   0.8840  & -0.3631\end{bmatrix}^T.\]
 In this case all the ranks involved were equal to four. So
$A_2$ and $A_3$ were computed already from ${\newnc{\M_3}(L,I_1\cap I_2)}$, after
an appropriate basis change $A_1$ (and the same $A_2,A_3$) was obtained
from ${\newnc{\M_3}(L,I_1)}$, and finally $A_4$ was computed from ${\newnc{\M_3}(L,I_2)}$.
\end{example}
\section{Trace Optimization of Noncommutative Sparse  Polynomials}
\label{sec:trace}

The aim of this section is to provide SDP relaxations allowing one to under-approximate the smallest trace of an nc polynomial on a semialgebraic set.
In Section~\ref{sec:sparse_tracial_gns}, we provide a sparse tracial representation for tracial linear functionals.
In Section~\ref{sec:unconstr_trace}, we address the unconstrained trace minimization problem. 
As in Section~\ref{sec:unconstr_eig}, we compute a lower bound on the smallest trace via SDP.
The constrained case is handled in Section~\ref{sec:constr_trace}, where we derive a hierarchy of lower bounds converging to the minimal trace, assuming that the quadratic module is archimedean and that RIP holds (Assumption~\ref{hyp:sparsityRIP}).
Most proofs are similar to the ones of eigenvalue problems addressed in Section~\ref{sec:eig}, so our treatment here is more concise.

We start this section by introducing useful notations about commutators and trace zero polynomials.
Given $g,h \in \RX$, the nc polynomial $[g,h] := g h - h g$ is called a \emph{commutator}. 
Two nc polynomials $g, h \in \RX$ are called \emph{cyclically equivalent} ($g  \cyc h$) if $g - h$ is a sum of commutators.
Given $S \subseteq \SymRX$ with corresponding quadratic module $M_S$ and truncated variant ${\cM(S)_d}$, one defines $\Theta_{S,d} := \{g \in \SymRX_{2d} : g \cyc h \text{ for some } h \in {\cM(S)_d} \}$ and $\Theta_S := \bigcup_{d \in \N} \Theta_{S,d}$. 
In this case, $\Theta_S$ stands for the \emph{cyclic quadratic module} generated by $S$ and $\Theta_{S,d}$ stands for the \emph{truncated cyclic quadratic module} generated by $S$.\\
For $S \subseteq \SymRX$  and ${D_S}$ as in~\eqref{eq:newDS} with the additional quadratic constraints~\eqref{eq:additional}, let us define $\Theta^k_{S,d} := \{g \in \SymRX_{2 d} : g \cyc h \text{ for some } h \in M^k_{S,d} \}$, $\Theta^{k}_S := \bigcup_{d \in \N} \Theta^{k}_{S,d}$, for all $k = 1,\dots,p$ and the sum
\begin{align}
\label{eq:sparse_cyclic_module}
\Theta^{\sparse}_{S,d} := \Theta^1_{S,d} + \dots + \Theta^p_{S,d} \,,
\end{align}
as well as $\Theta^{\sparse}_S := \bigcup_{d \in \N} \Theta^{\sparse}_{S,d}$.
If $S$ is empty, we drop the $S$ in the above notations.\\
%
An nc polynomial $g \in \SymRX$ is called a \emph{trace zero} nc polynomial if $\Trace (g(\underline{A})) = 0$, for all $\underline{A} \in \Sbb^n$. 
This is equivalent to $g \cyc 0$ (see e.g.~\cite[Proposition~2.3]{klep2008sums}).\\
\newnc{For a given nc polynomial $g$, the cyclic degree of $g$, denoted by $\cdeg g$, is  
the smallest degree of a polynomial cyclically equivalent to $g$.}
\subsection{Sparse Tracial Representations}
\label{sec:sparse_tracial_gns}
~\\
The next theorem  allows one to obtain a sparse tracial representation of a tracial linear functional, under the same  flatness and irreducibility conditions stated in Theorem~\ref{th:sparse_flat}. 
This is a sparse variant of~\cite[Theorem~1.71]{burgdorf16}.
\begin{theorem}
\label{th:sparse_flat_tracial}
Let $S  \subseteq \SymRX_{2d}$,  and assume that the semialgebraic set ${D_S}$ is as in~\eqref{eq:newDS} with the additional quadratic constraints~\eqref{eq:additional}. 
Let Assumption~\ref{hyp:sparsityRIP}(i) hold.
Set $\delta := \max \{ \lceil \deg (g)/2 \rceil : g \in S\cup {\{1\}} \}$.
Let $L : \RX_{2 d + 2 \delta} \to \R$ be a unital tracial linear functional satisfying $L(\Theta^{\sparse}_{S, d}) \subseteq \R^{\geq 0}$.
Assume that the flatness (H1) and irreducibility (H2) conditions of Theorem~\ref{th:sparse_flat} hold.
Then there are finitely many $n$-tuples $\underline{A}^{(j)}$ of symmetric matrices in ${D_S^r}$ for some $r  \in \N$,  and positive scalars $\lambda_j$ with $\sum_j \lambda_j = 1$, such that for all $f \in \RXI{1}_{2 d} + \dots + \RXI{p}_{2 d}$, one has:
\begin{align}
\label{eq:tracial}
L(f) =  \sum_j \lambda_j \Trace {f ( \underline{A}^{(j)} ) } \,.
\end{align}
\end{theorem}
\begin{proof}
As in Theorem~\ref{th:sparse_flat}, we perform the finite-dimensional GNS construction to obtain a tuple $\underline{A} \in {D_S^r}$, for some $r \in \N$, and unit vector $\v$ such that~\eqref{eq:tracial} holds.
To obtain the tracial representation, the proof is essentially the same as the one of~\cite[Theorem~1.71]{burgdorf16} and relies on the Wedderburn theorem, see e.g.~\cite[Chapter~1]{Lam13} for more details.
\end{proof}
\subsection{Unconstrained Trace Optimization with Sparsity}
\label{sec:unconstr_trace}
~\\
Given $f \in \SymRX$, the \emph{trace-minimum} of $f$ is obtained by solving the following  optimization problem
\begin{align}
\label{eq:tracemin}
\Trace_{\min}(f) := \inf \{\Trace f(\underline{A}) : \underline{A} \in \Sbb^n \} \,,
\end{align}
which is equivalent to
\begin{align}
\label{eq:tracemin2}
\Trace_{\min}(f) = \sup \{ a : \Trace (f - a) (\underline{A}) \geq 0 \,, \forall \underline{A} \in \Sbb^n \} \,,
\end{align}
If the cyclic degree of $f$ is odd, then $\Trace_{\min}(f) = -\infty$, thus let us assume that $2 d = \cdeg f$.
To approximate $\Trace_{\min}(f)$ from below, one considers the following relaxation:
\begin{align}
\label{eq:trace_dual}
\Trace_{\Theta}(f) = \sup \{ a : f - a \in \Theta_d \} \,,
\end{align}
whose dual is 
\begin{equation}
\label{eq:trace_primal}
\begin{aligned}
L_{\Theta}(f) := \inf\limits_{\newnc{L}} \quad  & \langle {\newnc{\M_d}(L)}, G_f \rangle  \\	
\text{s.t.} 
\quad & ({\newnc{\M_d}(L)})_{u,v} = ({\newnc{\M_d}(L)})_{w,z}  \,, \quad \text{for all } u^\star v \cyc w^\star z \,, \\
\quad & \newnc{L(1)} = 1 \,, \quad  {\newnc{\M_d}(L)} \succeq 0 \,, \\
\quad & \newnc{L : \RX_{2 d}  \to \R \,
\text{ linear}}\,,
\end{aligned}
\end{equation}
One has $ \Trace_{\Theta}(f) = L_{\Theta}(f) \leq \Trace_{\min}(f)$, where the inequality comes from~\cite[Lemma~5.2]{burgdorf16} and the equality results from the strong duality between SDP~\eqref{eq:trace_primal} and SDP~\eqref{eq:trace_dual}, see e.g.~\cite[Theorem~5.3]{burgdorf16} for a proof.
In addition, if the optimizer ${\newnc{\M_d}(L)}^\opt$ of SDP~\eqref{eq:trace_primal} satisfies the flatness condition, i.e., the linear functional underlying ${\newnc{\M_d}(L)}^\opt$ is $1$-flat (see Definition~\ref{def:flatextension}), then the above relaxations are exact and one has $ \Trace_{\Theta}(f) = L_{\Theta}(f) = \Trace_{\min}(f)$. This exactness result is stated in~\cite[Theorem~5.4]{burgdorf16}.

For a given nc polynomial $f = f_1 + \dots + f_p$, with $f_k \in \SymRXI{k}_{2d}$, for all $k=1,\dots,p$, we consider the following sparse variant of SDP~\eqref{eq:trace_primal}:
\begin{equation}
\label{eq:sparse_trace_primal}
\begin{aligned}
L^{\sparse}_{\Theta}(f) = \inf\limits_{\newnc{L}} \quad  & \sum_{k=1}^p \langle {\newnc{\M_d}(L,I_k)}, G_{f_k} \rangle \\	
\text{s.t.} 
\quad & ({\newnc{\M_d}(L,I_k)})_{u,v} = ({\newnc{\M_d}(L,I_k)})_{w,z}  \,, \quad \text{for all } u^\star v \cyc w^\star z \,,  \\
\quad & \newnc{L(1)} = 1 \,, \quad {\newnc{\M_d}(L,I_k)}  \succeq 0 \,,   \quad k=1,\dots,p \,,\\
\quad & \newnc{L : \RXI{1}_{2 d} + \dots + \RXI{p}_{2 d} \to \R \,
\text{ linear}}\,,
\end{aligned}
\end{equation}
whose dual is the sparse variant of SDP~\eqref{eq:trace_dual}:
\begin{equation}
\label{eq:sparse_trace_dual}
\begin{aligned}
\Trace^{\sparse}_{\Theta}(f) = \sup\limits_{\lambda} \quad  & \lambda  \\	
\text{s.t.} 
\quad & f  - \lambda \in \Theta^{\sparse}_d \,.
\end{aligned}
\end{equation}
Now, we are ready to state the sparse variant of~\cite[Theorem~5.3]{burgdorf16}.
\begin{theorem}
\label{th:sparse_trace_nogap}
Let $f \in \SymRX$ of degree $2 d$, with $f = f_1 + \dots + f_p$, $f_k \in \SymRXI{k}_{2d}$, for all $k=1,\dots,p$. 
There is no duality gap between SDP~\eqref{eq:sparse_trace_primal} and SDP~\eqref{eq:sparse_trace_dual}, namely $\Trace^{\sparse}_{\Theta}(f) = L^{\sparse}_{\Theta}(f)$.
\end{theorem}
\begin{proof}
The proof of strong duality is essentially the same as the one of Theorem~\ref{th:sparse_eig_nogap}. 
It relies on the closedness of the convex cone $\Theta^{\sparse}_d$ which comes from the closedness of $\Sigma^{\sparse}_d$, proved in Proposition~\ref{prop:sparseclosed}.
\end{proof}
As for unconstrained eigenvalue optimization, one can retrieve the solution of the initial trace minimization problem under the same assumptions as Theorem~\ref{th:sparse_flat}.
This is stated in the next proposition, which is the sparse variant of~\cite[Theorem~5.4]{burgdorf16}.
\begin{proposition}
\label{prop:sparse_trace_flat}
Let $f$ be as in Theorem~\ref{th:sparse_trace_nogap}, and  assume that SDP~\eqref{eq:sparse_trace_primal} admits an optimal  solution ${\newnc{\M_d}(L)}$.
If the linear functional $L$ underlying ${\newnc{\M_d}(L)}$ satisfies the flatness (H1) and irreducibility (H2) conditions stated in Theorem~\ref{th:sparse_flat}, then
\[ \Trace^{\sparse}_{\Theta}(f) = L^{\sparse}_{\Theta}(f) = \Trace_{\min}(f) \,. \]
\end{proposition}
\begin{proof}
The first equality comes from Theorem~\ref{th:sparse_trace_nogap}.
By Theorem~\ref{th:sparse_flat_tracial}, there exist finitely many $n$-tuples of symmetric matrices $\underline{A}^{(j)}$ and positive scalars $\lambda_j$ with $\sum_j \lambda_j = 1$ such that $L(f) =  \sum_j \lambda_j \Trace {f ( \underline{A}^{(j)} ) }$.
Since $L(f) = \sum_{k=1}^p \langle {\newnc{\M_d}(L,I_k)}, G_{f_k} \rangle = L^{\sparse}_{\Theta}(f) $ and $\Trace_{\min}(f) = \sum_j (\lambda_j \Trace_{\min}(f)) \leq \sum_j \lambda_j \Trace {f ( \underline{A}^{(j)} )} = L(f)$, one has $\Trace_{\min}(f) \leq L^{\sparse}_{\Theta}(f)$.
The desired result then follows from weak duality between SDP~\eqref{eq:sparse_trace_primal} and SDP~\eqref{eq:sparse_trace_dual}.
\end{proof}
In practice, Proposition~\ref{prop:sparse_trace_flat} allows one to derive an algorithm similar to the $\sparseeiggns$ procedure (described in Algorithm~\ref{algorithm:sparse_eig}) to find flat optimal solutions for the unconstrained trace problem.
\subsection{Constrained Trace Optimization with Sparsity}
\label{sec:constr_trace}
~\\
In this subsection, we provide the sparse tracial version of Lasserre's hierarchy to minimize the trace of a noncommutative polynomial on a semialgebraic set.
Given $f \in \SymRX$ and $S := \{g_1,\dots,g_{m} \} \subset \SymRX$ as in~\eqref{eq:DS}, let us define  $\Trace_{\min} (f, S)$ as follows:
\begin{align}
\label{eq:constr_trace}
\Trace_{\min}(f,S) := \inf \{\Trace f(\underline{A}) : \underline{A} \in {D_S} \} \,.
\end{align}
%
Since an infinite-dimensional Hilbert space does not admit a trace, we obtain lower bounds on the minimal trace by considering a particular subset of ${D_S^\infty}$.
This subset is obtained by restricting from the algebra of all bounded operators $\mathcal{B}(\mathcal{H})$ on a Hilbert space $\mathcal{H}$ to finite von Neumann algebras~\cite{Takesaki03} of type I and type II. 
We introduce $\Trace_{\min}(f,S)^{\II_1}$ as the trace-minimum of $f$ on ${D_S^{\II_1}}$.
This latter set is defined as follows (see~\cite[Definition~1.59]{burgdorf16}):
\begin{definition}
\label{def:DSII}
Let $\mathcal{F}$ be a type-$\II_1$-von Neumann algebra~\cite[Chapter~5]{Takesaki03}.
Let us define $\mathcal{D}_S^{\mathcal{F}}$ as the set of all tuples $\underline{A} = (A_1,\dots,A_n) \in \mathcal{F}^n$ making $s(\underline{A})$ a positive semidefinite operator for every $s \in S$. 
The von Neumann semialgebraic set ${D_S^{\II_1}}$ generated by $S$ is defined as
\[
{D_S^{\II_1}} := \bigcup_{\mathcal{F}} \mathcal{D}_S^{\mathcal{F}} \,,
\]
where the union is over all type-$\II_1$-von Neumann algebras with separable predual.
\end{definition} 
By~\cite[Proposition~1.62]{burgdorf16}, if $f \in \Theta_S$, then $\Trace f(\underline{A}) \geq 0$, for all $A \in {D_S}$ and $A \in {D_S^{\II_1}}$.
Since ${D_S}$ can be modeled by ${D_S^{\II_1}} $, one has $\Trace_{\min}(f,S)^{\II_1} \leq \Trace_{\min}(f,S)$.
With $d$ being defined as in Section~\ref{sec:constr_eig},
one can approximate $\Trace_{\min}(f,S)^{\II_1}$ from below via the following hierarchy of SDP programs, indexed by $s \geq d$: 
\begin{align}
\label{eq:constr_trace_dual}
\Trace_{\Theta,s}(f,S) = \sup \{ a : f - a \in \Theta_{S,d} \} \,,
\end{align}
whose dual is 
\begin{equation}
\label{eq:constr_trace_primal}
\begin{aligned}
L_{\Theta,s}(f,S) := \inf\limits_{\newnc{L}} \quad  & \langle {\newnc{\M_s}(L)}, G_f \rangle  \\	
\text{s.t.} 
\quad & ({\newnc{\M_s}(L)})_{u,v} = ({\newnc{\M_s}(L)})_{w,z}  \,, \quad \text{for all } u^\star v \cyc w^\star z \,, \\
\quad & \newnc{L(1)} = 1 \,, \\
\quad  & {\newnc{\M_s}(L)} \succeq 0 \,,  
\quad  {\newnc{\M_{s-d_j}}(g_j L)} \succeq 0 \,,    \quad j = 1,\dots,m \,,\\
\quad & \newnc{L : \RX_{2 d}  \to \R \,
 \text{ linear}}\,.
\end{aligned}
\end{equation}
If the quadratic module $M_S$ is archimedean, the resulting hierarchy of SDP programs provides a sequence of lower bounds $\Trace_{\Theta,s} (f,S)$ monotonically converging to $\Trace_{\min}(f,S)^{\II_1}$, see e.g. ~\cite[Corollary~3.5]{burgdorf16}.

Next, we present a sparse variant hierarchy of SDP programs providing a sequence of lower bounds $\Trace_{\Theta,s}^{\sparse} (f,S)$ monotonically converging to $\Trace_{\min}(f,S)^{\II_1}$.
Let $S \cup \{f\} \subseteq \SymRX$  and let ${D_S}$ be as in~\eqref{eq:newDS} with the additional quadratic constraints~\eqref{eq:additional}.  
Let us define the sparse variant of SDP~\eqref{eq:constr_trace_primal}, indexed by $s \geq d$:
\begin{equation}
\label{eq:sparse_constr_trace_primal}
\begin{aligned}
L^{\sparse}_{\Theta,s}(f,S) = \inf\limits_{\newnc{L}} \quad  & \sum_{k=1}^p \langle {\newnc{\M_s}(L,I_k)}, G_{f_k} \rangle \\	
\text{s.t.} 
\quad & ({\newnc{\M_s}(L,I_k)})_{u,v} = ({\newnc{\M_s}(L,I_k)})_{w,z}  \,, \quad \text{for all } u^\star v \cyc w^\star z \,,  \\
\quad & \newnc{L(1) = 1} \,, \\
\quad & {\newnc{\M_s}(L,I_k)} \succeq 0 \,, \quad k=1,\dots,p \,, \\ 
\quad &  {\newnc{\M_{s-d_j}}(g_jL,I_k)} \succeq 0 \,,  \quad j=1,\dots,m \,, \quad k=1,\dots,p \,, \\ 
\quad & \newnc{L : \RXI{1}_{2 d} + \dots + \RXI{p}_{2 d} \to \R \,
\text{ linear}}\,.
\end{aligned}
\end{equation}
whose dual is the sparse variant of SDP~\eqref{eq:constr_trace_dual}:
\begin{align}
\label{eq:sparse_constr_trace_dual}
\Trace_{\Theta,s}^{\sparse}(f,S) = \sup \{ a : f - a \in \Theta_{S,d}^{\sparse} \} \,,
\end{align}
With the same conditions as the ones assumed in Proposition~\ref{prop:constr_eigmin_slater} for constrained eigenvalue optimization, SDP~\eqref{eq:sparse_constr_trace_primal} admits strictly feasible solutions, so there is no duality gap between SDP~\eqref{eq:sparse_constr_trace_primal} and SDP~\eqref{eq:sparse_constr_trace_dual}.
The proof is the same since the constructed linear functional in Proposition~\ref{prop:constr_eigmin_slater} is tracial.
In order to prove convergence of the hierarchy of bounds given by the SDP~\eqref{eq:sparse_constr_trace_primal}-\eqref{eq:sparse_constr_trace_dual}, 
we need the following proposition, which is the sparse variant of~\cite[Proposition ~1.63]{burgdorf16}.
\begin{proposition}
\label{prop:sparse_tracial_psatz}
Let $S \cup \{f\} \subseteq \SymRX$  and let ${D_S}$ be as in~\eqref{eq:newDS} with the additional quadratic constraints~\eqref{eq:additional}. 
Let Assumption~\ref{hyp:sparsityRIP} hold. 
Then the following are equivalent:
\begin{itemize}
\item[(i)] $\Trace f(\underline{A}) \geq 0$ for all $\underline{A} \in {D_S^{\II_1}}$;
\item[(ii)] for all $\varepsilon > 0$, there exists $g \in {\cM(S)^{\sparse}}$ with $f + \varepsilon \cyc g$.
\end{itemize}
\end{proposition}
\begin{proof}
The implication (ii) $\implies$ (i) is trivial.
For the converse implication, let us fix $\varepsilon > 0$ such that the conclusion of (ii) does not hold.
By the Hahn-Banach separation theorem, there exists a linear functional $L : \SymRX \to \R$ with $L(f + \varepsilon) \leq 0 $ and $L({\cM(S)^{\sparse}}) \subseteq \R^{\geq 0}$.
As in Theorem~\ref{th:sparsePsatz},
the GNS construction 
leads to operator algebras $\mathcal A_k$, $\mathcal A_{jk}$ for $j,k=1,\ldots,p$ and $j\neq k$, with $\mathcal A_{jk}\subseteq\mathcal A_j,\mathcal A_k$. However, in this case the GNS construction yields tracial states on these, whence they are all finite von Neumann algebras. 
Now amalgamate in the category of von Neumann algebras (cf.~\cite{VDN})
to obtain a finite von Neumann algebra $\mathcal A$ with trace $\tau$
so that 
$\tau(f)\leq-\varepsilon<0$.
\end{proof}
Proposition~\ref{prop:sparse_tracial_psatz} implies the following convergence property.
\begin{corollary}
\label{coro:sparse_tracial_cvg}
Let $S \cup \{f\} \subseteq \SymRX$  and let ${D_S}$ be as in~\eqref{eq:newDS} with the additional quadratic constraints~\eqref{eq:additional}. 
Let Assumption~\ref{hyp:sparsityRIP} hold. 
Then
\[
\lim_{s \to \infty} \Trace_{\Theta,s}^{\sparse}(f,S) = \lim_{s \to \infty} L_{\Theta,s}^{\sparse}(f,S) = \Trace_{\min}(f,S)^{\II_1}
 \,.\]
\end{corollary}
\begin{proof}
By weak duality, one has $\Trace_{\Theta,s}^{\sparse}(f,S) \leq L_{\Theta,s}^{\sparse}(f,S) \leq \Trace_{\min}(f,S)^{\II_1}$.
In addition, Proposition~\ref{prop:sparse_tracial_psatz} implies that for each each $m \in \N$, there exists $s(m) \in \N$ such that $f - \Trace_{\min}(f,S)^{\II_1} + \frac{1}{m} \in \Theta_{S,s(m)}^{\sparse}$.
This implies that 
\[
\Trace_{\min}(f,S)^{\II_1} - \frac{1}{m} \leq  \Trace_{\Theta,s(m)}^{\sparse}(f,S) \,,
\]
yielding the desired conclusion.
\end{proof}
To extract solutions of constrained trace minimization problems, we rely on the following variant of Theorem~\ref{th:sparse_flat_tracial}. It is, in turn, the tracial analog of Theorem~\ref{th:sparse_flat}.
\begin{proposition}
\label{prop:sparse_constr_trace_flat}
Let $S  \subseteq \SymRX_{2d}$,  and assume that the semialgebraic set ${D_S}$ is as in~\eqref{eq:newDS} with the additional quadratic constraints~\eqref{eq:additional}. 
Let Assumption~\ref{hyp:sparsityRIP}(i) hold.
Set $\delta := \max \{ \lceil \deg (g)/2 \rceil : g \in S\cup {1} \}$.
Let ${\newnc{\M_s}(L)}$ be an optimal solution of SDP~\eqref{eq:sparse_trace_primal} with value $L_{\Theta,s}^{\sparse}(f,S)$, for $s \geq d + \delta$, such that $L$ satisfies the flatness (H1) and irreducibility (H2) conditions of Theorem~\ref{th:sparse_flat}.
Then there are finitely many $n$-tuples $\underline{A}^{(j)}$ of symmetric matrices in ${D_S^r}$ for some $r \in \N$, and positive scalars $\lambda_j$ with $\sum_j \lambda_j = 1$ such that
\[
L(f) =  \sum_j \lambda_j \Trace {f ( \underline{A}^{(j)} ) } \,.
\]
In particular, one has $\Trace_{\min}(f,S) = \Trace_{\min}(f,S)^{\II_1} = L_{\Theta,s}^{\sparse}(f,S)$.
\end{proposition}
As in the dense case~\cite[Algorithm~5.1]{burgdorf16}, one can rely on Proposition~\ref{prop:sparse_constr_trace_flat} to provide a randomized algorithm to look for flat optimal solutions for the constrained trace problem~\eqref{eq:constr_trace}.
\section{Numerical Experiments}
\label{sec:bench}

The aim of this section is to provide experimental comparison  between the bounds given by the dense relaxations 
(using $\eigmin$ under $\ncsostools$)
and the ones produced by our sparse variants.
For the sake of conciseness, we focus on minimal eigenvalue computation.\\
In Section~\ref{sec:benchsunc} we focus on the unconstrained case.  For a given nc polynomial $f$ of degree $2 d$, we compare the smallest eigenvalue $\lambda_{\min}(f) = \lambda_{\min,d}(f) = L_{\sohs,d}(f)$  computed via SDP~\eqref{eq:eigmin_primal} (or equivalently SDP~\eqref{eq:eigmin_dual}) with $\lambda^{\sparse}_{\min,d}(f) = L^{\sparse}_{\sohs,d}(f)$, computed via SDP~\eqref{eq:sparse_eigmin_primal} (or equivalently SDP~\eqref{eq:sparse_eigmin_dual}). \\
In Section~\ref{sec:benchsconstr} we focus on the {constrained} case.  
We compare the values of {$\lambda_s(f,S) = L_s(f,S)$}, obtained in the dense setting via SDP~\eqref{eq:constr_eigmin_primal} (or equivalently SDP~\eqref{eq:constr_eigmin_dual}), with the values of {$\lambda^{\sparse}_s(f) = L^{\sparse}_{s}(f)$}, obtained in the sparse setting via SDP~\eqref{eq:sparse_constr_eigmin_primal} (or equivalently SDP~\eqref{eq:sparse_constr_eigmin_dual}), for various sets of constraints $S$ and increasing values of the relaxation order $s$.

The resulting algorithm, denoted by $\eigminsparse$, is currently implemented in $\ncsostools$~\cite{cafuta2011ncsostools}. 
This software library is available within Matlab and interfaced with the SDP solver Mosek 8.1~\cite{moseksoft}, which turned out to yield better performance than SeDuMi 1.3~\cite{Sturm98usingsedumi}.
All numerical results were obtained  using a 
 cluster available at 
the Faculty of mechanical engineering,  University of Ljubljana, which has 30 TFlops computing performance. For our computations we used only one computing node which consisted of 2 Intel Xeon X5670 2,93GHz processors, each with 6 computing cores; 48 GB DDR3 memory; 500 GB hard drive.  
We ran Matlab in a plain (sequential) mode, without imposing any paralelization.

\subsection{Unconstrained Optimization}
\label{sec:benchsunc}
In Table~\ref{table:unc}, we report results obtained for minimizing the eigenvalue of the nc variants of the following functions:
\begin{itemize}
\item The chained singular function~\cite{conn1988testing}:
\[
f_{\text{cs}} := \sum_{i \in J} ( (X_i + 10 X_{i+1})^2 + 5 (X_{i+2} - X_{i+3})^2 + (X_{i+1} - 2 X_{i+2})^4 + 10 (X_{i} - X_{i+3})^4) \,,
\]
where $J = \{1,3,4,\dots,n-3 \}$ and $n$ is a multiple of 4.
In this case, one can choose $I_k = \{k,k+1,k+2,k+3 \}$ for all $k=1,\dots,n-3$ so that the associated sparsity pattern satisfies~\eqref{eq:RIP}.
\item The generalized Rosenbrock function~\cite{nash1984newton}:
\[
f_{\text{gR}} := 1 + \sum_{i=1}^{n-1} \Big(100 (X_{i+1} - X_{i}^2)^2 + (1 - X_{i+1})^2 \Big) \,.
\]
In this case, one can choose $I_k = \{k,k+1 \}$ for all $k=1,\dots,n-1$ so that the associated sparsity pattern satisfies~\eqref{eq:RIP}.
\end{itemize}
\begin{table}[!t]
\begin{center}
\caption{$\eigmin$ vs $\eigminsparse$ for unconstrained minimal eigenvalues of the chained singular and generalized Rosenbrock functions.}
\begin{tabular}{l|r|rrcr|rrcr}
\hline
\multirow{2}{*}{$f$} & \multirow{2}{*}{$n$}  & \multicolumn{4}{c|}{$\eigmin$} & \multicolumn{4}{c}{$\eigminsparse$} \\
& & $\msdp$ & $\nsdp$ & $\lambda_{\min,2}(f)$ & time (s) & $\msdp$ & $\nsdp$ & $\lambda^{\sparse}_{\min,2}(f)$ & time (s) \\
\hline
\multirow{6}{*}{$f_{\text{cs}}$} & 4 & 78 & 169 & 0 & 0.42 & 78 &169 & 0 & 0.37 \\
& 8  & 398 & 841 & 0 & 1.33 & 165 & 1323 & 0 & 3.69 \\
& 12 &974 &2025 & 0 & 4.35 & 298 & 2205 & 0 & 6.28 \\
& 16 &1806 &3721  & 0 &14.29  & 413 & 3087 & 0 & 9.18 \\
& 20 &2894 &5929  & 0 &52.47  & 537 & 3969 & 0 & 12.78 \\
& 24 &4238 &8649 & 0 &152.17 & 661 & 4851 & 0 & 17.65 \\  
\hline
\multirow{6}{*}{$f_{\text{gR}}$} & 10 & 200 & 400 & 0 &  0.56 & 95 & 441 & 0 & 1.39 \\
& 12 & 288 &576 & 0 & 0.81  & 117 & 539 & 0 & 1.78 \\
& 14 & 392 & 784 & 0 & 1.12  & 139 & 637 & 0 & 2.20 \\
& 16 & 512 &1024 & 0 & 1.46  & 161 & 735 & 0 & 2.67 \\
& 18 & 648 &1296 & 0 & 2.15  & 183 & 833 & 0 & 3.26 \\   
& 20 & 800 &1600  & 0 & 2.92  & 205 & 931 & 0 & 4.10 \\
\hline
\end{tabular}
\label{table:unc}
\end{center}
\end{table}
We compute bounds on the minimal eigenvalues of $f = f_{\text{cs}}$ for each $n \in \{4,\dots,24\}$ being a multiple of 4, and $f_{\text{gR}}$ for even values of $n \in \{2,\dots,20\}$.
For both functions, the minimal eigenvalue is 0.
We indicate in Table~\ref{table:unc} the data related to the semidefinite programs solved by Mosek.
For each value of $n$, $\msdp$ stands for the total number of constraints and $\nsdp$ stands for the total number of variables either of the SDP program~\eqref{eq:eigmin_primal} solved to compute $\lambda_{\min}(f)$ or the SDP program~\eqref{eq:sparse_eigmin_primal} solved to compute $\lambda_{\min,2}^{\sparse}(f)$.
As emphasized in the columns corresponding to $\msdp$, the size of the SDP programs can be significantly reduced after exploiting sparsity, which is consistent with Remark~\ref{rk:sparsevsdense}.
While the procedure $\eigmin$ does not take sparsity into account, it relies on the Newton chip method~\cite[\S2.3]{burgdorf16} to reduce the number of variables involved in the Hankel matrix from SDP~\eqref{eq:eigmin_primal}. 
This explains why $\nsdp$ is smaller for some values of $n$ (e.g. $n = 8$ for $f_{\text{cs}}$) when running $\eigmin$. 
However, the sparse procedure $\eigminsparse$ turns out to be very often more efficient to compute the minimal eigenvalue. 
So far, our $\eigminsparse$ procedure is limited by  the computational abilities of current SDP solvers (such as Mosek) to handle matrices with more constraints and variables than the ones obtained e.g. for the chained singular function at $n = 24$ (see the related values of $\msdp$ and $\nsdp$ in the corresponding column).
It turns out that exploiting the sparsity pattern yields SDP programs with significantly fewer variables than the ones obtained after running the Newton chip method. \\
In the column reporting timings, we indicate the time needed to prepare \emph{and} solve the SDP relaxation. 
For values of $n,d \gtrsim 8$, our current implementation in (interpreted) Matlab happens to be rather inefficient to construct the SDP problem itself, mainly because we rely on a naive nc polynomial arithmetic.
To overcome this computational burden, we plan to interface $\ncsostools$ with a \texttt C library implementing a more sophisticated monomial arithmetic. 
%
We also emphasize that for these unconstrained problems, each function is a sum of sparse hermitian squares, thus the sparse procedure $\eigminsparse$ always retrieves the same optimal value as the dense procedure $\eigmin$.
However, the bound computed via the sparse procedure can be a strict lower bound of the minimal eigenvalue, as shown in Lemma~\ref{lemma:nosparseHelton}.

\subsection{Constrained Optimization}
\label{sec:benchsconstr}
\begin{table}[!t]
\begin{center}
\caption{$\eigmin$ vs $\eigminsparse$ for minimal eigenvalue of the chained singular function on the nc polydisc $S_{\text{cs}}$.}
\begin{tabular}{r|rrcr|rrcr}
\hline
\multirow{2}{*}{$n$}  & \multicolumn{4}{c|}{$\eigmin$} & \multicolumn{4}{c}{$\eigminsparse$} \\
& $\msdp$ & $\nsdp$ & $\lambda_{2}(f_{\text{cs}},S_{\text{cs}})$ & time (s) & $\msdp$ & $\nsdp$ & $\lambda^{\sparse}_{2}(f_{\text{cs}},S_{\text{cs}})$ & time (s) \\
\hline
4 & 161 & 641 & 315.21 & 3.25   & 161    &  641  & 315.21& 2.95 \\
8 & 1009 & 6625 & 965.48 & 146.99 & 525  &1923   &  965.48 & 4.66 \\
12 & 3121 & 28705 & 1615.7 & 7891.6 & 889 &  3205 &  1615.7 & 7.43 \\
16 & \multicolumn{4}{c|}{$-$} &1253  &4487  & 2266.05 & 13.20 \\
20 & \multicolumn{4}{c|}{$-$} & 1617  & 5769 &2916.32 & 18.50 \\
24 & \multicolumn{4}{c|}{$-$} & 1981   & 7051 & 3566.56 & 26.38\\
\hline
\end{tabular}
\label{table:polydisc}
\end{center}
\end{table}
In Table~\ref{table:polydisc}, we report results obtained for minimizing the eigenvalue of the nc chained singular function on the semialgebraic set $S_{\text{cs}} := \{1 - X_1^2,\dots, 1 - X_n^2, X_1 - 1/3,\dots,X_n - 1/3 \}$ for $n \in \{4,8,12,16,20,24\}$.
Since $f$ has degree $4$, it follows from~\cite[Corollary~4.18]{burgdorf16} that it is enough to solve SDP~\eqref{eq:sparse_eigmin_primal} with optimal value $\lambda_{2}(f,S_{\text{cs}})$ to compute the minimal eigenvalue $\lambda_{\min} (f,S_{\text{cs}})$. 
For the experiments described in Table~\ref{table:polydisc}, we cannot rely on the Newton chip method as in the unconstrained case. 
Thus the dense procedure $\eigmin$ suffers from a severe computational burden for $n > 10$; the symbol ``$-$'' in a column entry indicates that
the calculation did not finish in a couple of hours. 
As already observed before for the unconstrained case, the sparse procedure $\eigminsparse$ performs much better than $\eigmin$.
Surprisingly, $\eigminsparse$ yields the same bounds as $\eigmin$ at the minimal relaxation order $s = 2$, for all values of $n \leq 10$. 

As shown in Example~\ref{ex:nosparseEigBall}, there is no guarantee to obtain the above mentioned convergence behavior in a systematic way.
We consider randomly generated cubic $n$-variate polynomials  $f_{\text{rand}}$ satisfying Assumption~\ref{hyp:sparsityRIP} with $I_k = \{k,k+1,k+2\}$, for all $k=1,\dots,n-2$.
The corresponding nc polyball is given by $\Bbb^{\sparse} := \{ 1 - X_1^2 - X_2^2 - X_3^2,\dots, 1 - X_{n-2}^2 - X_{n-1}^2 - X_n^2\}$.
In Table~\ref{table:polyball}, we report results obtained for minimizing the eigenvalue of $f_{\text{rand}}$ on  $\Bbb^{\sparse}$, for each value of $n \in \{4,\dots,10\}$.
Here again, the sparse procedure $\eigminsparse$ yields better performance than $\eigmin$.
Moreover, the sparse bound obtained for each $n \leq 10$ at minimal relaxation order $s = 2$ already gives an accurate approximation of the optimal bound provided by the dense procedure.
We emphasize that the value of the third order relaxation obtained with the sparse procedure is almost equal to the optimal bound.
In addition, the dense procedure cannot handle to solve the minimal order relaxation for $n > 10$, while we can always obtain a lower bound of the eigenvalue with $\eigminsparse$.
\begin{table}[!t]
\begin{center}
\caption{$\eigmin$ vs $\eigminsparse$ for minimal eigenvalue of random cubic polynomials on the nc polyball $S = \Bbb^{\sparse}$.}
\begin{tabular}{r|rrcr|rrrcr}
\hline
\multirow{2}{*}{$n$}  & \multicolumn{4}{c|}{$\eigmin$} & \multicolumn{5}{c}{$\eigminsparse$} \\
& $\msdp$ & $\nsdp$ & $\lambda_{2}(f_{\text{rand}},S)$ & time (s) & $s$ & $\msdp$ & $\nsdp$ & $\lambda^{\sparse}_{s}(f_{\text{rand}},S)$ & time (s) \\
\hline
\multirow{2}{*}{4}  & \multirow{2}{*}{71}  & \multirow{2}{*}{491}  & \multirow{2}{*}{-53.64}  & \multirow{2}{*}{3.31} & 2 & 79 & 370 & -53.72 & 1.18\\
&  &  &  &  & 3 & 729 & 3538 & -53.64 & 12.64 \\
\hline
\multirow{2}{*}{6}  & \multirow{2}{*}{239}  & \multirow{2}{*}{2045}  & \multirow{2}{*}{-142.52}  & \multirow{2}{*}{26.79} & 2 & 179 & 740 & -142.62 & 2.33 \\
&  &  &  &  & 3 & 1535 & 7076 & -142.52 & 29.52\\
\hline
\multirow{2}{*}{8}  & \multirow{2}{*}{559}  & \multirow{2}{*}{5815}  & \multirow{2}{*}{-165.89}  & \multirow{2}{*}{171.30} & 2 & 279 & 1110 & -166.32 & 3.73 \\
&  &  &  &  & 3 & 2341 & 10614 & -165.91 & 62.70 \\
\hline    
\multirow{2}{*}{10}  & \multirow{2}{*}{1079}  & \multirow{2}{*}{13289}  & \multirow{2}{*}{-199.62}  & \multirow{2}{*}{857.95} & 2 & 379 & 1480 &-200.51 & 5.43\\
&  &  &  &  & 3 & 3147 &  14152   &  -199.66 & 139.22 \\
\hline
\multirow{2}{*}{11}  & \multirow{2}{*}{1429}  & \multirow{2}{*}{18985}  & \multirow{2}{*}{-180.39}  & \multirow{2}{*}{ 2111.26} & 2 & 429 &  1665  & -180.93  & 6.58\\
&  &  &  &  & 3 & 3550 &  15921  & -180.40 & 209.73 \\ 
\hline
\multirow{2}{*}{12}  & \multirow{2}{*}{-}  & \multirow{2}{*}{-}  & \multirow{2}{*}{-}  & \multirow{2}{*}{ -} & 2 & 479    & 1850  &  -385.89 & 7.82\\
&  &  &  &  & 3 &  3953    & 17690   &  -384.87 &  289.12\\
\hline
\multirow{2}{*}{16}  & \multirow{2}{*}{-}  & \multirow{2}{*}{-}  & \multirow{2}{*}{-}  & \multirow{2}{*}{ -} & 2 & 679   & 2590    & -344.31 & 15.46\\
&  &  &  &  & 3 & 5565  & 24766 & -342.15 &  975.43\\ 
\hline
\multirow{2}{*}{20}  & \multirow{2}{*}{-}  & \multirow{2}{*}{-}  & \multirow{2}{*}{-}  & \multirow{2}{*}{ -} & 2 & 879  & 3330 & -504.36 &  31.41\\
&  &  &  &  & 3 & 7177  & 31842 & -503.02 &  2587.61\\
\hline
\end{tabular}
\label{table:polyball}
\end{center}
\end{table}

\section{Conclusion and Perspectives}
\label{sec:concl}

We have presented a sparse variant of Putinar's Positivstellensatz for  positive noncommutative polynomials, yielding a converging hierarchy of semidefinite relaxations for eigenvalue and trace optimization. 
We also designed a general algorithm to extract solutions of such sparse problems, thanks to a sparse variant of the Gelfand-Naimark-Segal construction
and amalgamation properties of operator algebras.
Experimental results obtained with $\ncsostools$ prove that one can obtain accurate lower bounds via these semidefinite relaxations in an efficient way.

An obvious direction of further research is to investigate whether and how one can benefit from sparsity exploitation in other application fields, for instance to compute certified approximations of quantum graph parameters or maximum violation bounds of Bell inequalities in quantum information theory. 
\newnc{We also intend to design a noncommutative analog of the recently developed procedures exploiting monomial term sparsity \cite{wang2,wang3,wang4}.}

We have proved that there is no sparse analog of the Helton-McCullough Sums of Squares theorem.
Thus, another interesting track of research is to look for alternative representations of sparse positive polynomials, e.g., representations involving noncommutative rational functions.

Apart from sparsity, we also intend to pursue research efforts to take into account other properties of structured noncommutative polynomials, such as symmetry.
%

\begin{thebibliography}{BEGFB94}

\bibitem[AL11]{anjos2011handbook}
Miguel~F. Anjos and Jean~B. Lasserre, editors.
\newblock {\em Handbook on Semidefinite, Conic and Polynomial Optimization},
  volume 166.
\newblock Springer Science \& Business Media, 2011.

\bibitem[Bar02]{Bar02}
Alexander Barvinok.
\newblock {\em A course in convexity}, volume~54 of {\em Graduate Studies in
  Mathematics}.
\newblock American Mathematical Society, Providence, RI, 2002.

\bibitem[BCKP13]{nctrace}
Sabine Burgdorf, Kristijan Cafuta, Igor Klep, and Janez Povh.
\newblock The tracial moment problem and trace-optimization of polynomials.
\newblock {\em Math. Program.}, 137(1-2, Ser. A):557--578, 2013.

\bibitem[BEGFB94]{BEFB94}
Stephen Boyd, Laurent El~Ghaoui, Eric Feron, and Venkataramanan Balakrishnan.
\newblock {\em Linear Matrix Inequalities in System and Control Theory},
  volume~15 of {\em Studies in Applied Mathematics}.
\newblock {SIAM}, Philadelphia, PA, June 1994.

\bibitem[BKP16]{burgdorf16}
Sabine Burgdorf, Igor Klep, and Janez Povh.
\newblock {\em Optimization of polynomials in non-commuting variables}.
\newblock SpringerBriefs in Mathematics. Springer, [Cham], 2016.

\bibitem[Bla78]{Amalgam78}
Bruce~E. Blackadar.
\newblock Weak expectations and nuclear {$C^{\ast} $}-algebras.
\newblock {\em Indiana Univ. Math. J.}, 27(6):1021--1026, 1978.

\bibitem[BMV75]{bessis1975monotonic}
Daniel Bessis, Pierre Moussa, and Matteo Villani.
\newblock Monotonic converging variational approximations to the functional
  integrals in quantum statistical mechanics.
\newblock {\em J. Math. Phys.}, 16(11):2318--2325, 1975.

\bibitem[BO13]{BOns}
Gr\'{e}gory Berhuy and Fr\'{e}d\'{e}rique Oggier.
\newblock {\em An introduction to central simple algebras and their
  applications to wireless communication}, volume 191 of {\em Mathematical
  Surveys and Monographs}.
\newblock American Mathematical Society, Providence, RI, 2013.

\bibitem[BP93]{blair1993introduction}
Jean R.~S. Blair and Barry Peyton.
\newblock An introduction to chordal graphs and clique trees.
\newblock In {\em Graph theory and sparse matrix computation}, volume~56 of
  {\em IMA Vol. Math. Appl.}, pages 1--29. Springer, New York, 1993.

\bibitem[Bre14]{bresar2014}
Matej Bresar.
\newblock {\em Introduction to noncommutative algebra}.
\newblock Universitext. Springer, Cham, 2014.

\bibitem[CGT88]{conn1988testing}
Andrew~R. Conn, Nicholas I.~M. Gould, and Philippe~L. Toint.
\newblock Testing a class of methods for solving minimization problems with
  simple bounds on the variables.
\newblock {\em Math. Comp.}, 50(182):399--430, 1988.

\bibitem[CKP11]{cafuta2011ncsostools}
Kristijan Cafuta, Igor Klep, and Janez Povh.
\newblock N{CSOS}tools: a computer algebra system for symbolic and numerical
  computation with noncommutative polynomials.
\newblock {\em Optim. Methods Softw.}, 26(3):363--380, 2011.

\bibitem[CKP12]{cafuta2012constrained}
Kristijan Cafuta, Igor Klep, and Janez Povh.
\newblock Constrained polynomial optimization problems with noncommuting
  variables.
\newblock {\em SIAM J. Optim.}, 22(2):363--383, 2012.

\bibitem[CLMP20]{chen}
Tong Chen, Jean-Bernard Lasserre, Victor Magron, and Edouard Pauwels.
\newblock {Semialgebraic Optimization for Bounding Lipschitz Constants of ReLU
  Networks}.
\newblock {\em arXiv preprint arXiv:2002.03657}, 2020.

\bibitem[FKMN01]{fukuda2001exploiting}
Mituhiro Fukuda, Masakazu Kojima, Kazuo Murota, and Kazuhide Nakata.
\newblock Exploiting sparsity in semidefinite programming via matrix
  completion. {I}. {G}eneral framework.
\newblock {\em SIAM J. Optim.}, 11(3):647--674, 2000/01.

\bibitem[GdLL18]{Gribling18}
Sander Gribling, David de~Laat, and Monique Laurent.
\newblock Bounds on entanglement dimensions and quantum graph parameters via
  noncommutative polynomial optimization.
\newblock {\em Math. Program.}, 170(1, Ser. B):5--42, 2018.

\bibitem[GdLL19]{Gribling19}
Sander Gribling, David de~Laat, and Monique Laurent.
\newblock Lower bounds on matrix factorization ranks via noncommutative
  polynomial optimization.
\newblock {\em Found. Comput. Math.}, to appear 2019.

\bibitem[GNS07]{grimm2007note}
David Grimm, Tim Netzer, and Markus Schweighofer.
\newblock A note on the representation of positive polynomials with structured
  sparsity.
\newblock {\em Archiv der Mathematik}, 89(5):399--403, 2007.

\bibitem[Hel02]{Helton02}
J.~William Helton.
\newblock ``{P}ositive'' noncommutative polynomials are sums of squares.
\newblock {\em Ann. of Math. (2)}, 156(2):675--694, 2002.

\bibitem[HKM12]{hkmConvex}
J.~William Helton, Igor Klep, and Scott McCullough.
\newblock The convex {P}ositivstellensatz in a free algebra.
\newblock {\em Adv. Math.}, 231(1):516--534, 2012.

\bibitem[HLL09]{gloptipoly3}
Didier Henrion, Jean-Bernard Lasserre, and Johan L\"{o}fberg.
\newblock Glopti{P}oly 3: moments, optimization and semidefinite programming.
\newblock {\em Optim. Methods Softw.}, 24(4-5):761--779, 2009.

\bibitem[HLS09]{HLS09vol}
Didier Henrion, Jean-Bernard Lasserre, and Carlo Savorgnan.
\newblock Approximate volume and integration for basic semialgebraic sets.
\newblock {\em SIAM Rev.}, 51(4):722--743, 2009.

\bibitem[HM04]{Helton04}
J.~William Helton and Scott~A. McCullough.
\newblock A {P}ositivstellensatz for non-commutative polynomials.
\newblock {\em Trans. Amer. Math. Soc.}, 356(9):3721--3737, 2004.

\bibitem[Jam70]{jameson1970ordered}
Graham Jameson.
\newblock Ordered linear spaces.
\newblock In {\em Ordered linear spaces}, pages 1--39. Springer, 1970.

\bibitem[Jos16]{Josz16}
C\'{e}dric Josz.
\newblock {\em {Application of polynomial optimization to electricity
  transmission networks}}.
\newblock Theses, {Universit{\'e} Pierre et Marie Curie - Paris VI}, July 2016.

\bibitem[Kri64]{krivineanneaux}
Jean-Louis Krivine.
\newblock Anneaux pr\'{e}ordonn\'{e}s.
\newblock {\em J. Analyse Math.}, 12:307--326, 1964.

\bibitem[KS08]{klep2008sums}
Igor Klep and Markus Schweighofer.
\newblock Sums of {H}ermitian squares and the {BMV} conjecture.
\newblock {\em J. Stat. Phys.}, 133(4):739--760, 2008.

\bibitem[Lam13]{Lam13}
Tsit-Yuen Lam.
\newblock {\em A first course in noncommutative rings}, volume 131.
\newblock Springer Science \& Business Media, 2013.

\bibitem[Las01]{Las01sos}
Jean-Bernard Lasserre.
\newblock Global optimization with polynomials and the problem of moments.
\newblock {\em SIAM J. Optim.}, 11(3):796--817, 2000/01.

\bibitem[Las06]{Las06}
Jean-Bernard Lasserre.
\newblock Convergent {SDP}-relaxations in polynomial optimization with
  sparsity.
\newblock {\em SIAM J. Optim.}, 17(3):822--843, 2006.

\bibitem[Lau09a]{Lau09}
Monique Laurent.
\newblock Matrix completion problems.
\newblock In Christodoulos~A. Floudas and Panos~M. Pardalos, editors, {\em
  Encyclopedia of Optimization}, pages 1967--1975. Springer, 2009.

\bibitem[Lau09b]{Laurent:Survey}
Monique Laurent.
\newblock Sums of squares, moment matrices and optimization over polynomials.
\newblock In {\em Emerging applications of algebraic geometry}, volume 149 of
  {\em IMA Vol. Math. Appl.}, pages 157--270. Springer, New York, 2009.

\bibitem[Lax58]{lax1957differential}
Peter~D. Lax.
\newblock Differential equations, difference equations and matrix theory.
\newblock {\em Comm. Pure Appl. Math.}, 11:175--194, 1958.

\bibitem[LR05]{laurent2005semidefinite}
Monique Laurent and Franz Rendl.
\newblock Semidefinite programming and integer programming.
\newblock {\em Handbooks in Operations Research and Management Science},
  12:393--514, 2005.

\bibitem[LS04]{lieb2004equivalent}
Elliott~H. Lieb and Robert Seiringer.
\newblock Equivalent forms of the {B}essis-{M}oussa-{V}illani conjecture.
\newblock {\em J. Statist. Phys.}, 115(1-2):185--190, 2004.

\bibitem[LTY17]{lasserre2017bounded}
Jean-Bernard Lasserre, Kim-Chuan Toh, and Shouguang Yang.
\newblock A bounded degree {SOS} hierarchy for polynomial optimization.
\newblock {\em EURO J. Comput. Optim.}, 5(1-2):87--117, 2017.

\bibitem[Mag18]{toms18}
Victor Magron.
\newblock {Interval Enclosures of Upper Bounds of Roundoff Errors Using
  Semidefinite Programming}.
\newblock {\em ACM Trans. Math. Softw.}, 44(4):41:1--41:18, June 2018.

\bibitem[McC01]{McCullSOS}
Scott McCullough.
\newblock Factorization of operator-valued polynomials in several non-commuting
  variables.
\newblock {\em Linear Algebra Appl.}, 326(1-3):193--203, 2001.

\bibitem[MCD17]{toms17}
Victor Magron, George Constantinides, and Alastair Donaldson.
\newblock Certified roundoff error bounds using semidefinite programming.
\newblock {\em ACM Trans. Math. Software}, 43(4):Art. 34, 31, 2017.

\bibitem[MKKK10]{MKKK10}
Kazuo Murota, Yoshihiro Kanno, Masakazu Kojima, and Sadayoshi Kojima.
\newblock A numerical algorithm for block-diagonal decomposition of matrix
  {$*$}-algebras with application to semidefinite programming.
\newblock {\em Jpn. J. Ind. Appl. Math.}, 27(1):125--160, 2010.

\bibitem[MLM20]{sparseReznick}
N.~H.~A. Mai, J.-B. Lasserre, and V.~Magron.
\newblock {A sparse version of Reznick's Positivstellensatz}.
\newblock {\em arXiv preprint arXiv:2002.05101}, 2020.
\newblock Submitted.

\bibitem[Mos]{moseksoft}
{The MOSEK optimization software}.
\newblock \url{http://www.mosek.com/}.

\bibitem[MP05]{mccullough2005noncommutative}
Scott McCullough and Mihai Putinar.
\newblock Noncommutative sums of squares.
\newblock {\em Pacific J. Math.}, 218(1):167--171, 2005.

\bibitem[Nas84]{nash1984newton}
Stephen~G. Nash.
\newblock Newton-type minimization via the {L}\'{a}nczos method.
\newblock {\em SIAM J. Numer. Anal.}, 21(4):770--788, 1984.

\bibitem[NFF{\etalchar{+}}03]{nakata2003exploiting}
Kazuhide Nakata, Katsuki Fujisawa, Mituhiro Fukuda, Masakazu Kojima, and Kazuo
  Murota.
\newblock Exploiting sparsity in semidefinite programming via matrix
  completion. {II}. {I}mplementation and numerical results.
\newblock {\em Math. Program.}, 95(2, Ser. B):303--327, 2003.

\bibitem[Nie14]{NieRand14}
Jiawang Nie.
\newblock The {$\mathcal{A}$}-truncated {$K$}-moment problem.
\newblock {\em Found. Comput. Math.}, 14(6):1243--1276, 2014.

\bibitem[NPA08]{navascues2008convergent}
Miguel Navascu{\'e}s, Stefano Pironio, and Antonio Ac{\'\i}n.
\newblock A convergent hierarchy of semidefinite programs characterizing the
  set of quantum correlations.
\newblock {\em New J. Phys.}, 10(7):073013, 2008.

\bibitem[NT14]{netzer2014hyperbolic}
Tim Netzer and Andreas Thom.
\newblock Hyperbolic polynomials and generalized {C}lifford algebras.
\newblock {\em Discrete Comput. Geom.}, 51(4):802--814, 2014.

\bibitem[PNA10]{pironio2010convergent}
Stefano Pironio, Miguel Navascu{\'e}s, and Antonio Ac{\'\i}n.
\newblock Convergent relaxations of polynomial optimization problems with
  noncommuting variables.
\newblock {\em SIAM J. Optim.}, 20(5):2157--2180, 2010.

\bibitem[Put93]{Putinar1993positive}
Mihai Putinar.
\newblock Positive polynomials on compact semi-algebraic sets.
\newblock {\em Indiana Univ. Math. J.}, 42(3):969--984, 1993.

\bibitem[PV09]{pal2009quantum}
K\'{a}roly~F. P\'{a}l and Tam\'{a}s V\'{e}rtesi.
\newblock Quantum bounds on {B}ell inequalities.
\newblock {\em Phys. Rev. A (3)}, 79(2):022120, 12, 2009.

\bibitem[Rez78]{Reznick78}
Bruce Reznick.
\newblock Extremal {PSD} forms with few terms.
\newblock {\em Duke Math. J.}, 45(2):363--374, 1978.

\bibitem[RTAL13]{Riener13Symmetries}
Cordian Riener, Thorsten Theobald, Lina~Jansson Andr\'{e}n, and Jean-Bernard
  Lasserre.
\newblock Exploiting symmetries in {SDP}-relaxations for polynomial
  optimization.
\newblock {\em Math. Oper. Res.}, 38(1):122--141, 2013.

\bibitem[SIG98]{skelton1997unified}
Robert~E. Skelton, Tetsuya Iwasaki, and Karolos~M. Grigoriadis.
\newblock {\em A unified algebraic approach to linear control design}.
\newblock The Taylor \& Francis Systems and Control Book Series. Taylor \&
  Francis, Ltd., London, 1998.

\bibitem[Sta13]{stahl2013proof}
Herbert~R. Stahl.
\newblock Proof of the {BMV} conjecture.
\newblock {\em Acta Math.}, 211(2):255--290, 2013.

\bibitem[Stu99]{Sturm98usingsedumi}
Jos~F. Sturm.
\newblock Using {S}e{D}u{M}i 1.02, a {MATLAB} toolbox for optimization over
  symmetric cones.
\newblock {\em Optim. Methods Softw.}, 11/12(1-4):625--653, 1999.

\bibitem[Tak03]{Takesaki03}
Masamichi Takesaki.
\newblock {\em Theory of operator algebras. {III}}, volume 127 of {\em
  Encyclopaedia of Mathematical Sciences}.
\newblock Springer-Verlag, Berlin, 2003.
\newblock Operator Algebras and Non-commutative Geometry, 8.

\bibitem[TTT03]{SDPT3}
Reha~H. T\"{u}t\"{u}nc\"{u}, Kim-Chuan Toh, and Michael~J. Todd.
\newblock Solving semidefinite-quadratic-linear programs using {SDPT}3.
\newblock {\em Math. Program.}, 95(2, Ser. B):189--217, 2003.

\bibitem[TWLH19]{tacchi2019exploiting}
Matteo Tacchi, Tillmann Weisser, Jean-Bernard Lasserre, and Didier Henrion.
\newblock Exploiting sparsity for semi-algebraic set volume computation.
\newblock {\em preprint arXiv:1902.02976}, 2019.

\bibitem[VDN92]{VDN}
Dan-Virgil Voiculescu, Kenneth~J. Dykema, and Alexandru Nica.
\newblock {\em Free random variables}, volume~1 of {\em CRM Monograph Series}.
\newblock American Mathematical Society, Providence, RI, 1992.

\bibitem[Voi85]{Voi83}
Dan-Virgil Voiculescu.
\newblock Symmetries of some reduced free product {$C^\ast$}-algebras.
\newblock In {\em Operator algebras and their connections with topology and
  ergodic theory ({B}u{s}teni, 1983)}, volume 1132 of {\em Lecture Notes in
  Math.}, pages 556--588. Springer, Berlin, 1985.

\bibitem[Wit15]{wittek2015algorithm}
Peter Wittek.
\newblock Algorithm 950: {N}cpol2sdpa-sparse semidefinite programming
  relaxations for polynomial optimization problems of noncommuting variables.
\newblock {\em ACM Trans. Math. Software}, 41(3):Art. 21, 12, 2015.

\bibitem[WKK{\etalchar{+}}09]{WakiKKMS08}
Hayato Waki, Sunyoung Kim, Masakazu Kojima, Masakazu Muramatsu, and Hiroshi
  Sugimoto.
\newblock Algorithm 883: sparse{POP}---a sparse semidefinite programming
  relaxation of polynomial optimization problems.
\newblock {\em ACM Trans. Math. Software}, 35(2):Art. 15, 13, 2009.

\bibitem[WKKM06]{Waki06}
Hayato Waki, Sunyoung Kim, Masakazu Kojima, and Masakazu Muramatsu.
\newblock Sums of squares and semidefinite program relaxations for polynomial
  optimization problems with structured sparsity.
\newblock {\em SIAM J. Optim.}, 17(1):218--242, 2006.

\bibitem[WLT18]{weisser2018sparse}
Tillmann Weisser, Jean-Bernard Lasserre, and Kim-Chuan Toh.
\newblock Sparse-{BSOS}: a bounded degree {SOS} hierarchy for large scale
  polynomial optimization with sparsity.
\newblock {\em Math. Program. Comput.}, 10(1):1--32, 2018.

\bibitem[WML19]{wang2}
Jie Wang, Victor Magron, and Jean-Bernard Lasserre.
\newblock {TSSOS: A Moment-SOS hierarchy that exploits term sparsity}.
\newblock {\em arXiv preprint arXiv:1912.08899}, 2019.

\bibitem[WML20]{wang3}
Jie Wang, Victor Magron, and Jean-Bernard Lasserre.
\newblock {Chordal-TSSOS: a moment-SOS hierarchy that exploits term sparsity
  with chordal extension}.
\newblock {\em arXiv preprint arXiv:2003.03210}, 2020.

\bibitem[WMLM20]{wang4}
Jie Wang, Victor Magron, Jean-Bernard Lasserre, and Ngoc Hoang~Anh Mai.
\newblock {CS-TSSOS: Correlative and term sparsity for large-scale polynomial
  optimization}.
\newblock {\em arXiv preprint arXiv:2005.02828}, 2020.

\bibitem[YFK03]{sdpa}
Makoto Yamashita, Katsuki Fujisawa, and Masakazu Kojima.
\newblock Implementation and evaluation of {SDPA} 6.0 (semidefinite programming
  algorithm 6.0).
\newblock volume~18, pages 491--505. 2003.
\newblock The Second Japanese-Sino Optimization Meeting, Part II (Kyoto, 2002).

\end{thebibliography}
\newcommand{\etalchar}[1]{$^{#1}$}

\end{document}